%% file: ExtrapolEllipt.tex
\documentclass[parskip=half]{scrartcl}
\usepackage{amsmath,amsthm,amssymb}
\usepackage{mathtools}

\usepackage{graphicx}
\usepackage{color}
\usepackage[noadjust]{cite}
\usepackage[shortlabels]{enumitem}
\usepackage[german,english]{babel}
\usepackage[format=plain,justification=centerlast,font=small]{caption}

\newcommand{\eps}{\varepsilon}
\newcommand{\embeds}{\hookrightarrow}

\newcommand{\cA}{\mathcal{A}}
\newcommand{\cC}{\mathcal{C}}
\newcommand{\cE}{\mathcal{E}}
\newcommand{\cH}{\mathcal{H}}
\newcommand{\cI}{\mathcal{I}}
\newcommand{\cL}{\mathcal{L}}
\newcommand{\cLiso}{\cL_{\mathrm{iso}}}
\newcommand{\cM}{\mathcal{M}}

\newcommand{\cP}{\mathcal{P}}
\newcommand{\cU}{\mathcal{U}}

\newcommand{\bfW}{\mathbf{W}}
\newcommand{\bfX}{\mathbf{X}}

\newcommand{\Diri}{D}
\newcommand{\Neumann}{\Gamma}

\newcommand{\fd}{\mathfrak{d}}
\newcommand{\reaction}{r}

\newcommand{\field}[1]{\mathbb{#1}}
\newcommand{\N}{\field{N}}
\newcommand{\R}{\field{R}}
\newcommand{\C}{\field{C}}

\newcommand{\dx}{\,\mathrm{d}x}
\newcommand{\dy}{\,\mathrm{d}y}
\newcommand{\dd}{\,\mathrm{d}}

\DeclareMathOperator{\tr}{tr}      
\DeclareMathOperator{\dom}{dom}    
\DeclareMathOperator{\supp}{supp}  
\DeclareMathOperator{\dive}{div}   
\DeclareMathOperator{\dist}{dist}   
\DeclareMathOperator{\itr}{i-tr}
\renewcommand{\Re}{\operatorname{Re}}

\newcommand{\blangle}{\bigl\langle}
\newcommand{\brangle}{\bigr\rangle}

\newcommand{\defn}{\coloneqq}

\numberwithin{equation}{section}

\newenvironment{keywords}{{\footnotesize\textbf{Keywords}: \quad}\footnotesize}{

}
\newenvironment{AMS}{{\footnotesize\textbf{AMS classification}
  (2020):\quad}\footnotesize}{

  }

\usepackage[dvipsnames]{xcolor}
\usepackage[colorlinks=true,allcolors={RoyalPurple}]{hyperref}
\usepackage[capitalise,nameinlink]{cleveref}

\newtheorem{theorem}{Theorem}[section]
\newtheorem{lemma}[theorem]{Lemma}
\newtheorem{corollary}[theorem]{Corollary}
\newtheorem{proposition}[theorem]{Proposition}

\theoremstyle{definition}
\newtheorem{definition}[theorem]{Definition}
\newtheorem{assu}[theorem]{Assumption}
\newtheorem{rem}[theorem]{Remark}

\theoremstyle{remark}

\begin{document}

\title{Extrapolated elliptic regularity and application to the van Roosbroeck system
  of semiconductor equations} \author{Hannes Meinlschmidt, Joachim
  Rehberg}

\maketitle

\input{ExtrapolEllipt-content}

\end{document}

%% file: ExtrapolEllipt-content.tex
\begin{abstract}
  In this paper we present a general \emph{extrapolated elliptic
    regularity} result for second order differential operators in
  divergence form on fractional Sobolev-type spaces of negative order
  $X^{s-1,q}_D(\Omega)$ for $s > 0$ small, including mixed boundary
  conditions and with a fully nonsmooth geometry of $\Omega$ and the
  Dirichlet boundary part $D$. We expect the result to find
  applications in the analysis of nonlinear parabolic equations, in
  particular for quasilinear problems or when treating coupled systems
  of equations. To demonstrate the usefulness of our result, we give a
  new proof of local-in-time existence and uniqueness for the
  van Roosbroeck system for semiconductor devices which is
  much simpler than already established proofs.
\end{abstract}


\section{Introduction}

Let $\Omega \subset \R^d$ be a bounded domain with nonsmooth
boundary of which the set $\Diri$ is a subset. Let further $\rho$
be a bounded measurable uniformly-a.e.\ positive definite
coefficient matrix defined on $\Omega$, and let
$-\nabla\cdot\rho\nabla$ be the associated second-order
differential operator in divergence form. One may consider this
operator as the principal part of a possibly more general linear
differential operator. Assume that for some $q \in (1,\infty)$ the
following optimal elliptic regularity property holds true:
\begin{equation}
-\nabla\cdot\rho\nabla u \in W^{-1,q}_\Diri(\Omega)  \quad \implies \quad
u \in W^{1,q}_\Diri(\Omega),\label{eq:intro-assu}
\end{equation}
where
$W^{-1,q}_\Diri(\Omega) \defn (W^{1,q'}_\Diri(\Omega))^\star$,
the space of antilinear functionals on $W^{1,q'}_\Diri(\Omega)$,
and the subscript $\Diri$ refers to zero boundary trace on
$\Diri$. Of course, the probably best known optimal elliptic
regularity result is the Hilbert space case $q=2$
where~\eqref{eq:intro-assu} is always true under our assumptions
on $\rho$ if $\Diri$ is sufficiently large such that
$\mathbf{1} \notin W^{1,2}_\Diri(\Omega)$. There are countless
works extending this result also to~\eqref{eq:intro-assu} for
the integrability scale $q>2$; we mention
exemplarily~\cite{jonsson,haller,groeger,BMMM} where also mixed
boundary conditions and \emph{nonsmooth} data $\Omega$ and
$\rho$ are treated.  In this paper we establish an optimal
elliptic regularity result for a \emph{differentiability} scale
such as
$H^{s-1,q}_\Diri(\Omega) = (H^{1-s,q'}_\Diri(\Omega))^\star$
starting from~\eqref{eq:intro-assu}. More precisely, we show
that if~\eqref{eq:intro-assu} holds true for some
$q \in (1,\infty)$ and if there is $\tau > 0$ such that each
component $\rho_{ij}$ of the coefficient matrix function is a
multiplier on the Bessel potential space
$H^{\tau,q}_\Diri(\Omega)$, then there exists a number
$\bar s \in (0,\tau]$ such that
\begin{equation}\label{eq:intro-extrapolate-result}
-\nabla\cdot\rho\nabla u \in H^{s-1,q}_\Diri(\Omega)  \quad \implies \quad
u \in H^{1+s,q}_\Diri(\Omega)
\end{equation}
for $s \in (-\bar s,\bar s)$. (We give precise definitions of function
spaces and assumptions in \cref{sec:preliminaries}
below.)  The result is obtained from~\eqref{eq:intro-assu} by an
\emph{extrapolation} technique: We establish that
$W^{1,q}_D(\Omega)$ and $W^{-1,q}_D(\Omega)$ are ``interior
points'' in the interpolation scale of Bessel potential spaces
and the dual scale, and that $-\nabla\cdot\rho\nabla$ is
compatible with that scale. Then the Sneiberg extrapolation
theorem~(\cite{sneib}) gives the result. This is what is meant
by the titular \emph{extrapolated elliptic regularity}. We
remark that~\eqref{eq:intro-assu} is already nontrivial to
have, in particular if $q$ is not around $2$. The
extrapolation technique and recent interpolation results also
allow to obtain~\eqref{eq:intro-extrapolate-result} for the
Slobodetskii scale $W^{1+s,q}_D(\Omega)$ and
$W^{s-1,q}_D(\Omega)$ as a byproduct. Moreover, we in fact
establish~\eqref{eq:intro-extrapolate-result} not only for pure
second-order operators but also for such including lower order
terms and in particular boundary forms arising from Robin
boundary conditions. Thanks to a quantitative version of the
Sneiberg theorem which was recently established
in~\cite{AuschEg}, we can also provide
property~\eqref{eq:intro-extrapolate-result} and bounds on the
inverse operators uniform in the given data. Such uniform
results are extremely useful in the treatment of nonautonomous
or even quasilinear evolution equations,
cf.~\cite{MR16,pruess,PruessSchnaubelt}.

Note that while inferring~\eqref{eq:intro-extrapolate-result}
from~\eqref{eq:intro-assu} may feel like an ``expected'' result, the
necessary groundwork behind the reasoning is highly nontrivial since
we suppose essentially no smoothness in the data at all. This is in
particular the case since~\eqref{eq:intro-extrapolate-result} for
$q>d$ (ambient space dimension) is of elevated interest to us for
conceptual reasons in the treatment of abstract nonlinear evolution
equations. Let us take this for granted at the moment; we explain it
in detail in the next subsection of this introduction. It is known
since the sixties that in the present case of nonsmooth data, one in
general cannot expect $q$ in the assumed~\eqref{eq:intro-assu} to be
larger than a prescribed number $\bar q > 2$, see for
example~\cite{meyers,savare,elschner}. (Due to Sobolev embeddings, the
size of $s$ in~\eqref{eq:intro-extrapolate-result} is thus also
limited in the general case.) This makes already the
assumption~\eqref{eq:intro-assu} sensible for $q>d=3$.  In fact, to
the best of the authors' knowledge, the only comparable results
for~\eqref{eq:intro-extrapolate-result} which include mixed boundary
conditions and nonsmooth data are~\cite{joch}, for a relatively
restricted geometry, and~\cite{haller}, with very general
geometry. Both works are limited to $q$ close to $2$
in~\eqref{eq:intro-extrapolate-result}, starting from the Lax Milgram
result. Another conceptual obstacle is the availability of a suitable
interpolation theory framework for $H^{\sigma,q}_D(\Omega)$ spaces
also for $q \neq 2$. Fortunately, both issues have been resolved
recently:

\begin{enumerate}[(i)]
\item In~\cite{disser}, the authors collect a rich setting of
  geometric constellations for $\Omega,\Diri$ and the coefficient
  functions $\rho$ under which~\eqref{eq:intro-assu} is satisfied for
  $q>d=3$. This includes a wide array of quite nonsmooth situations
  occurring in real-world problems.
\item In their seminal paper~\cite{BM19}, Bechtel and Egert
  establish a comprehensive interpolation theory for the
  Bessel potential (and Sobolev Slobodetskii) scale in an
  extremely general geometric setup. Their work extends
  previously known results under similar geometric assumptions
  in~\cite{EHDT} for the Hilbert scale corresponding to
  $q=2$. (In fact, these older results were used
  in~\cite{haller}.)
\end{enumerate}

We explicitly point out that both works are highly
nontrivial and in turn rest on other difficult
results. (See~\cite[Introduction]{disser} for more background.)
Let us also note that already~\eqref{eq:intro-assu} for $q>d$
itself has turned out to be an extremely valuable and well
suited---one might even say, indispensable---property in the
treatment of nonlinear and/or coupled systems of evolution
equations with highly nonsmooth data arising in real-life
problems, see e.g.~\cite{vanroos3d,horst,MMR17a,MMR17b}. We next
motivate why we need also the optimal regularity
result~\eqref{eq:intro-extrapolate-result} for $q>d$ in the
fractional Sobolev scales.

\subsubsection*{Motivation and real-world example: semiconductor equations}

One of the main areas where optimal elliptic regularity results
like~\eqref{eq:intro-extrapolate-result} are needed is the analysis of
nonlinear evolution equations. We give a real-world example in
\cref{sec:semiconductor} below by considering the van
  Roosbroeck system of semiconductor equations, but we expect many
more applications to be susceptible to similar reasoning. For now,
consider for example the following abstract Fokker-Planck type
evolution equation posed in some Banach space $X$ over some time
interval $J$ as a model problem:
\begin{equation}
  \label{eq:model-equation-quad}
  \partial_t u - \nabla\cdot\mu\nabla u = \nabla\cdot u\mu
  \nabla \bigl(-\nabla \cdot 
  \rho\nabla\bigr)^{-1}f + |\nabla u|^2 + g \quad
  \text{in}~X, 
\end{equation}
where $\mu$ is another matrix coefficient function of the same quality
as $\rho$, while $f$ and $g$ are appropriate data, the latter e.g.\
coming from inhomogeneous Neumann boundary conditions. Such model
equations are related to the \emph{viscous Hamilton-Jacobi equation}
or the \emph{deterministic KPZ equation}; we exemplarily refer
to~\cite{BGK04,GGK03,Porretta15}. One may imagine having obtained this
abstract equation from eliminating the second equation in the abstract
system
\begin{align*}
    \partial_t u - \nabla\cdot\mu\nabla u -\nabla\cdot u\mu
  \nabla w &=  |\nabla u|^2 + g, \\ -\nabla \cdot 
  \rho\nabla w & = f.
\end{align*}
It turns out that in this situation, in order to deal with the
quadratic nonlinearity in~\eqref{eq:model-equation-quad} in the
framework of classical theory for semilinear
equations~(\cite[Ch.~6]{pazy}), $X$ should be chosen an as
interpolation space of the form
$[L^q(\Omega),W^{-1,q}_\Diri(\Omega)]_{1-s}$ with parameters $q>d$ and
$s \in (0,1-\frac{d}q)$; this was observed in~\cite[Sect.~6]{RoJo},
see also~\cite[Sect.~4.1]{vanroos3d}. We fix $X$ to be such a space
for the following. Note that
$X=[L^q(\Omega),W^{-1,q}_\Diri(\Omega)]_{1-s}$ can indeed be
identified with the (dual) Bessel potential space
$H^{s-1,q}_\Diri(\Omega)$ under very mild assumptions on $\Omega$ and
the geometry of $\Diri$. We also mention that dealing with the
quadratic nonlinearity in~\eqref{eq:model-equation-quad} does
\emph{not} require explicit knowledge of the domains of the elliptic
operators in $X$. This however changes when we consider the drift term
for $u$ where we assume that $f$ is in general not more regular than
generic elements of $H^{s-1,q}_\Diri(\Omega)$---e.g.\ also arising
from inhomogeneous Neumann boundary conditions---, because then we
further have to assure that the operators
$ \nabla \cdot u \mu \nabla (-\nabla \cdot \rho \nabla )^{-1}$
in~\eqref{eq:abstract-model-equation} are bounded ones when considered
on $X$ in order to obtain a self-consistent abstract formulation. More
precisely the domain of $-\nabla \cdot \rho \nabla$ in $X$ must be
continuously embedded into the domain of
$\nabla \cdot u(t) \mu \nabla $ in $X$ for $t \in J$. The optimal case
and thus the natural candidate for the domain of definition for these
elliptic operators in $X=H^{s-1,q}_D(\Omega)$ is the space
$H^{1+s,q}_D(\Omega)$, cf.\ e.g.~\cite[Ch.~5.7.1]{Triebel}. While the
actual domains of the operators $\nabla \cdot u(t) \rho \nabla $ in
$X$ will in general not coincide with $H^{1+s,q}_D(\Omega)$ and vary
with $t$ without further assumptions, one easily observes that
$H^{1+s,q}_D(\Omega)$ is indeed the largest space which will embed
continuously into every such $t$-dependent domain. Thus, in general,
$\nabla \cdot u \mu \nabla (-\nabla \cdot \rho \nabla)^{-1}$ will be
bounded on $X$ exactly when the optimal elliptic regularity
result~\eqref{eq:intro-extrapolate-result} holds true.  In that sense,
wellposedness of the reduced problem~\eqref{eq:model-equation-quad}
boils down exactly to the availability of the optimal regularity
property~\eqref{eq:intro-extrapolate-result} for $q>d$.

In the second part of the paper, we rigorously follow the above
roadmap and prove local-in-time existence and uniqueness for the
van Roosbroeck system for semiconductor devices using the
extrapolated elliptic regularity result. The van Roosbroeck
system describes the evolution of the triple $(u_1,u_2,\varphi)$
of unknowns---representing electron- and hole densities and
electrostatic potential---during the (finite) time interval
$J = (0,T)$ by the following system of coupled equations,
consisting of the \emph{Poisson equation}
\begin{subequations}
  \label{vanRoos}
  \begin{equation}
    \label{Poisson-eq}
    \begin{aligned}
      -\dive\left( \varepsilon \nabla \varphi \right) & =
      {\fd} + u_1 - u_2 &\quad&
      \text{in}~J\times\Omega, \\
      \varphi & = {\varphi}_\Diri &\quad&
      \text{on}~J\times \Diri, \\
      {\nu}\cdot{ \left( \varepsilon \nabla \varphi \right) } +
      \varepsilon_{\Gamma} \varphi & ={\varphi}_\Gamma &\quad&
      \text{on}~J\times\Neumann,
    \end{aligned}
  \end{equation}
  so a quasi-static elliptic equation with inhomogeneous Dirichlet and
  Robin boundary data, and, for $k=1,2$, the \emph{current-continuity
    equations}
  \begin{equation}
  \label{CuCo-eq}
    \begin{aligned}
      \partial_t u_k - \dive j_k
      &= r^\Omega(u,\varphi)  &\quad& \text{in}~J\times (\Omega \setminus \Pi)\\
      u_k &=U_k &\quad&\text{on}~J\times\Diri,
      \\
      {\nu}\cdot{j_k} &= {r}^\Gamma(u,\varphi) &\quad&\text{on}~J\times\Neumann,
      \\
      [{\nu}\cdot{j_k}] &= {r}^\Pi(u,\varphi)&\quad&\text{on}~J\times\Pi, \\
      u_k(0) & = u_k^0 &\quad& \text{on}~\Omega,
    \end{aligned}
  \end{equation}
  with the currents
  \begin{equation} \label{eq:curr-dens} j_k = \mu_k\bigl(\nabla u_k  +(-1)^k u_k
    \nabla \varphi \bigr).
  \end{equation}
\end{subequations}
The latter equations are nonlinear Fokker-Planck equations with
inhomogenenous mixed boundary conditions and a jump condition along a
surface $\Pi \subset \Omega$. Here, $\Omega \subset \R^3$ is a bounded
domain representing a semiconductor device, $\nu$ its unit outer
normal at $\partial\Omega$ and the latter is decomposed into a
Dirichlet part $\Diri$ and a Neumann/Robin part
$\Neumann \defn \partial\Omega\setminus \Diri$. We pose only very low
regularity assumptions on the geometry of $\Diri,\Neumann$ and $\Pi$
which will cover nearly all practical situations arising in realistic
devices. This is made more precise in \cref{sec:semiconductor}
below, where the model and the involved quantities are also explained
in detail.  We refer to the introduction of~\cite{vanroos3d} for a
comprehensive collection of related literature. In fact, the van
Roosbroeck system~\eqref{vanRoos} was treated under similar
assumptions recently in~\cite{vanroos3d}; however, the analysis there
is quite involved since the system need be reformulated ``globally''
in the \emph{quasi Fermi levels}. We are able to provide a much
simpler treatment basing on the extrapolated elliptic regularity
result~\eqref{eq:intro-extrapolate-result} by
solving~\eqref{Poisson-eq} for $\varphi$ in dependence of $u$ and
inserting this dependence into~\eqref{CuCo-eq}, thereby reducing the
current-continuity equations to equations in $u$ alone. Let us explain
the principal idea and its connection to the above.

Suppose that we have formally solved~\eqref{Poisson-eq} for $\varphi$
in dependence of $u$ and consider the (reduced) \emph{recombination}
functions $u\mapsto (r^\Omega,r^\Neumann,r^\Pi)(u,\varphi(u))$
in~\eqref{CuCo-eq}. Then an abstract reduced formulation
of~\eqref{CuCo-eq} would be
\begin{equation}
  \label{eq:abstract-model-equation}
  \partial_t u - \nabla\cdot\mu\nabla u = \nabla \cdot u\mu
  \nabla\bigl(-\nabla\cdot\eps\nabla  +
  \tr_\Neumann^*\eps_\Neumann\tr_\Neumann\bigr)^{-1}(\fd+u+\tr_\Neumann^*\varphi_\Neumann) +  f(u) 
\end{equation}
where the nonlinearity $f$ represents the reduced recombination
functions, $\tr_\Neumann$ is the trace operator onto $\Neumann$, and
we have ignored the multiple components of $u$ and the Dirichlet
boundary data in the equations for the sake of exposition at this
point. This equation is of the same type as the model
problem~\eqref{eq:model-equation-quad}. In fact, it turns out that the
commonly used \emph{Avalanche generation} model for $r^\Omega$
contained in $f$ in~\eqref{eq:abstract-model-equation} in a sense
behaves quite similarly to the quadratic gradient nonlinearity
in~\eqref{eq:abstract-model-equation}, see
\cref{rem:similarity-quad}, and all the arguments from the above
motivation apply. In the case of~\eqref{vanRoos}, we indeed need
property~\eqref{eq:intro-extrapolate-result} also for the second order
operator including the boundary form
$ \tr_\Neumann^*\eps_\Neumann\tr_\Neumann$ corresponding to the Robin
boundary conditions.

\subsection*{Outline}

The first part of this work first establishes the necessary groundwork
for all of the following in \cref{sec:preliminaries}. We prove
the extrapolated elliptic regularity result in full generality with
lower order terms together with the necessary preparations as
announced in the introduction in \cref{sec:extr-ellipt-regul}
(\cref{thm:invert-extrapol}). In the second part,
\cref{sec:semiconductor}, the elliptic regularity results are
then put to work for providing a proof of (local-in-time) existence
and uniqueness of solutions to the Van Roosbroeck
  system~\eqref{vanRoos} which is considerably easier than having to
deal with one big macroscopic standard model for the electron/hole
flux within the semiconductor as done in~\cite{vanroos3d}
(\cref{t-formulat}). We restrict ourselves to Boltzmann
statistics. This is done only for technical simplicity, since already
here all crucial effects which we want to make visible are already
present. We note that one can carry out an analogous program for the
quasilinear system arising in case of Fermi-Dirac statistics, see
\cref{r-concl}.

\section{Preliminaries}\label{sec:preliminaries}

All notation used in this paper is considered as standard or
self-explanatory by the authors. Up to
\cref{sec:semiconductor}, where we treat the van
  Rooesbroeck system~\eqref{vanRoos}, we consider a general space
dimension $d \geq 2$. Starting from \cref{sec:semiconductor},
we fix $d=3$.

\subsection{Assumptions}\label{sec:assumptions}

We pose the following general assumptions on the underlying spatial
domain $\Omega \subseteq \R^d$ and its boundary part
$D \subseteq \partial\Omega$. They are supposed to hold true from now
on for the rest of this work. We recall the following notion, refering
to e.g.~\cite{JW84}:

\begin{definition}[Regular set]\label{def:ahlfors}
  Let $0 < N \leq d$. The set $\Lambda \subseteq \R^d$ is called
  \emph{$N$-set} or \emph{$N$-regular}, if there exist constants
  $0<c\leq C$ such that
  \begin{equation}\label{eq:N-set}
    cr^N \leq \cH_N(B_r(x) \cap \Lambda) \leq Cr^N \qquad (x \in \Lambda,~r \in (0,1]).
  \end{equation}
\end{definition}

\begin{rem}
  For $N=d$, the upper estimate requirement in~\eqref{eq:N-set}
  is trivial. Thus, the \emph{interior thickness condition},
  so that there exists $\gamma > 0$ such that
  \begin{equation}\label{eq:ICTfull}
    {|B_r(x) \cap \Lambda | } \ge \gamma |B_r(x)| \qquad (x \in \Lambda,~r \in (0,1]),\tag{ICT}
  \end{equation}
  becomes a sufficient condition for $\Lambda$ to be $d$-regular. In
  fact, the interior thickness condition~\eqref{eq:ICTfull} can
  equivalently be required only for
  $x \in \partial\Lambda$~(\cite[Lem.~3.2]{Bec20}). In the latter
  form, the property is also called \emph{$d$-thick} by some authors,
  see e.g.~\cite{BMMM}. There will be yet another \emph{thickness}
  assumption for the treatment of the semiconductor equations in
  \cref{ass:geometry-extend}.
  \label{rem:ahlfors-thickness}
\end{rem}

\begin{assu}[Geometry] \label{a-geometry} The set
  $\Omega \subset \R^d$ is a bounded domain satisfying the
  \emph{interior thickness
    condition}~\eqref{eq:ICTfull}. (Equivalently: $\Omega$ is a
  $d$-set.) Moreover, the boundary
  $\partial\Omega$ has the following properties:
  \begin{enumerate}[(i)]    
  \item $D \subseteq \partial\Omega$ is a closed
    $(d-1)$-set.
  \item There are Lipschitz coordinate charts available around
    $\overline{\partial\Omega \setminus D}$, that is, for every
    $x \in \overline{\partial\Omega \setminus D}$, there is an open
    neighborhood $\cU$ of $x$ and a bi-Lipschitz mapping
    $\phi_x \colon \cU \to (-1,1)^d$ such that $\phi_x(x) = 0$ and
    $\phi_x(\cU\cap\Omega) = (-1,0) \times (-1,1)^{d-1}$.
  \end{enumerate}
\end{assu}

\begin{rem} \label{r-d-1} From $(d-1)$-regularity of $\Diri$ and the
  Lipschitz charts for $\overline{\partial \Omega \setminus D}$ we
  obtain that the whole boundary $\partial\Omega$ is also a
  $(d-1)$-set. See~\cite[Ex.~2.4/2.5]{BM19}.
\end{rem}

\subsection{Function spaces}\label{sec:function-spaces}

For $s \in \R$ and $p \in (1,\infty)$, let $H^{s,p}(\R^d)$ denote the
Bessel potential spaces. We mention that
$H^{-s,p'}(\R^d) = H^{s,p}(\R^d)^\star$.  We further note that for
$k \in \N_0$, the classical Sobolev space of $k$th order
$W^{k,p}(\R^d)$ coincides with $H^{k,p}(\R^d)$ up to equivalent norms.
See e.g.~\cite[Ch.~2.3.3\&2.6.1]{Triebel}.

\begin{definition}[Sobolev-Slobodetskii spaces]\label{def:slobo-spaces}
  Let $p \in (1,\infty)$ and $s > 0$ not an integer. Write
  $s = k + \sigma$ with $k \in \N_0$ and $\sigma \in (0,1)$.  Then the
  space $W^{s,p}(\R^d)$ is given by the normed vector space of
  functions $u \in L^p(\R^d)$ for which
  \begin{equation*}
    \|u\|_{W^{s,p}(\R^d)} := \|u\|_{W^{k,p}(\R^d)} +
    \left(\sum_{i=1}^d\iint_{\R^d\times\R^d} \frac {|\partial_i^k u(x)-\partial_i^ku(y)|^p}{|x-y|^{d + \sigma p}} \dx \dy\right)^{1/p} < \infty.
  \end{equation*}
  Moreover, we define $W^{-s,p'}(\R^d) \defn W^{s,p}(\R^d)^\star$, the
  space of antilinear continuous functionals on $W^{s,p}(\R^d)$.
\end{definition}

Let $X \in \{H,W\}$ for the remainder of this section. We next turn to
traces.

\begin{proposition}[{\cite[Thms.~VI.1\&VII.1]{JW84}}]
  \label{prop:restriction-d-1}
  Let $E \subset \R^d$ be a $(d-1)$-set and let
  $s \in (\frac1p,1+\frac1p)$ with $p \in (1,\infty)$. Then the
  trace operator $\tr_E$ defined by
  \begin{equation*}
    (\tr_E u)(x) :=  \lim_{r\searrow0}\frac1{|B_r(x)|}
    \int_{B_r(x)} u \quad (x \in E)
  \end{equation*}
  maps $X^{s,p}(\R^d)$ continuously into $L^p(E;\cH_{d-1})$.
\end{proposition}

\begin{definition}[Function spaces with zero trace]
  \label{def:func-spaces-dirichlet}
  Let $E \subset \R^d$ be a $(d-1)$-set and let
  $s \in (\frac1p,1+\frac1p)$ with $p \in (1,\infty)$. Then we define
  $X^{s,p}_E(\R^n) := \ker \tr_E$ in $X^{s,p}(\R^n)$.
\end{definition}

The versions of the spaces $X^{s,p}$ and $X^{s,p}_E$ on $\Omega$ are
defined as quotient spaces corresponding to restriction to $\Omega$ of
their $\R^d$ versions as follows:

\begin{definition}[Function spaces on $\Omega$]
  \label{def:func-spaces-domain}
  Let $p \in (1,\infty)$ and $s > 0$.
  \begin{enumerate}[(i)]
  \item We define $X^{s,p}(\Omega)$ to be the factor space of
    restrictions to $\Omega$ of $X^{s,p}(\R^d)$, equipped with the
    natural quotient norm. Moreover,
    $X^{-s,p'}(\Omega) \defn X^{s,p}(\Omega)^\star$.
  \item Let now $s \in (\frac1p,1 + \frac {1}{p})$ and let
    $E \subseteq \overline\Omega$ be a $(d-1)$-set. Then, as before,
    we define $X^{s,p}_E(\Omega)$ to be the factor space of
    restrictions to $\Omega$ of $X_E^{s,p}(\R^d)$, equipped with the
    natural quotient norm. Moreover,
    $X^{-s,p'}_E(\Omega) \defn X^{s,p}_E(\Omega)^\star$.
  \end{enumerate}
\end{definition}

\begin{rem}
  \label{rem:restriction-def-embedding}
  The definition of the spaces $X^{s,p}(\Omega)$
  as factor spaces of restrictions implies that these spaces inherit
  the usual Sobolev-type embeddings between them from their full-space
  analogues.
\end{rem}

\begin{rem}
  \label{rem:intrinsic-slobodetskii}
  Let $s \in (0,1)$. Then it is well known that since $\Omega$
  satisfies~\eqref{eq:ICTfull}, the factor space $W^{s,p}(\Omega)$
  agrees with the space $W^{s,p}_*(\Omega)$ defined
  intrinsically by the set of all functions $u \in L^p(\Omega)$ such
  that
  \begin{equation*}
    \|u\|_{W^{s,p}_*(\Omega)} := \|u\|_{L^p(\Omega)} +    
    \left(\iint_{\Omega\times\Omega} \frac {|u(x)-u(y)|^p}{|x-y|^{d
          + sp}} \dx \dy\right)^{1/p} < \infty
  \end{equation*}
  up to equivalent norms. (See~\cite[Thm.~V.1]{JW84}). Moreover, very
  recently it was shown in~\cite{Bec20} that if
  $E \subseteq \partial\Omega$ is $(d-1)$-regular and $\Omega$
  satisfies the interior thickness condition~\eqref{eq:ICTfull} for
  $x \in \partial\Omega\setminus E$, then $W^{s,p}_E(\Omega)$
  coincides with the intrinsically given
  $W^{s,p}_{*}(\Omega) \cap L^p(\Omega,\dist_E^{-sp})$, also up to
  equivalent norms.
\end{rem}

We next quote interpolation results from~\cite{BM19} for
\emph{symmetric} interpolation where both involved spaces carry
partially vanishing trace. This result and its dual variant below will
be used for the extrapolated elliptic regularity result in \cref{sec:extr-ellipt-regul}.

\begin{proposition}[{Interpolation~\cite[Thm.~1.2]{BM19}}]
  \label{prop:interpolation-zero-trace}
  Let $p_i \in (1,\infty)$ and $s_i \in (\frac1{p_i},1+\frac1{p_i})$
  for $i=1,2$. Set
  $\frac1{p_\theta} = \frac{1-\theta}{p_0} + \frac{\theta}{p_1}$ and
  $s_\theta = (1-\theta)s_0 + \theta s_1$. Let further
  $E \subseteq \overline\Omega$ be a $(d-1)$-set. Then, up to
  equivalent norms, we have
  \begin{equation}\label{eq:complex-interpolation}
    \bigl[X^{s_0,p_0}_E(\Omega),X^{s_1,p_1}_E(\Omega)\bigr]_\theta =
    X^{s_\theta,p_\theta}_E(\Omega)
  \end{equation}
  and
  \begin{equation}\label{eq:real-interpolation}
    \bigl(X^{s_0,p_0}_E(\Omega),X^{s_1,p_1}_E(\Omega)\bigr)_{\theta,p_\theta} =
    W^{s_\theta,p_\theta}_E(\Omega),    
  \end{equation}
  with the following exceptions: if $s_\theta = 1$
  in~\eqref{eq:real-interpolation}, then we must already have
  $s_0=s_1=1$; moreover, $X=W$ is permitted
  in~\eqref{eq:complex-interpolation} only if either all or none of
  $s_0,s_1,s_\theta$ are 1.
\end{proposition}

\begin{corollary}
  \label{cor:interpolation-dual}
  Adopt the assumptions of
  \cref{prop:interpolation-zero-trace}. Then, up to
  equivalent norms, we have
  \begin{equation*}
    \bigl[X^{-s_0,p_0}_E(\Omega),X^{-s_1,p_1}_E(\Omega)\bigr]_\theta =
    X^{-s_\theta,p_\theta}_E(\Omega)
  \end{equation*}
  and
  \begin{equation*}
    \bigl(X^{-s_0,p_0}_E(\Omega),X^{-s_1,p_1}_E(\Omega)\bigr)_{\theta,p_\theta} =
    W^{-s_\theta,p_\theta}_E(\Omega),    
  \end{equation*}
  with the exceptions as in
  \cref{prop:interpolation-zero-trace}.
\end{corollary}

\begin{proof}
  The assertions follow from the corresponding ones in
  \cref{prop:interpolation-zero-trace} by general duality
  properties of the interpolation functors, see
  e.g.~\cite[Ch.~1.11.3]{Triebel}. Before we validate the assumptions
  there, let us note that the present corollary is an assertion about
  \emph{anti-dual} spaces, whereas the cited result is about ordinary
  dual spaces. However, we can recover the anti-dual case from the
  dual one by means of the retraction-coretraction
  theorem~(\cite[Ch.~1.2.4]{Triebel}) using the mapping
  $\psi \mapsto [f \mapsto \langle \psi,\overline f\rangle]$ both as
  the retraction and coretraction between anti-dual and dual space.

  Now let us turn to the assumptions in~\cite[Ch.~1.11.3]{Triebel}:
  First, $X^{s_0,p_0}_E(\Omega) \cap X^{s_1,p_1}_E(\Omega)$ is dense
  in $X^{s_i,p_i}_E(\Omega)$ for $i=1,2$. This can be seen as follows:
  For all $p \in (1,\infty)$ and $s \in (\frac1p,1+\frac1p)$, the
  spaces $X^{s,p}_E(\R^d)$ are complemented subspaces of
  $X^{s,p}(\R^d)$ by virtue of a $(s,p)$-uniform projection $\cP$ as
  shown in~\cite[Lem.~3.1]{BM19}. But
  $X^{s_0,p_0}(\R^d) \cap X^{s_1,p_1}(\R^d)$ is dense in
  $X^{s_i,p_i}(\R^d)$, hence
  $\cP\bigl(X^{s_0,p_0}(\R^d) \cap X^{s_1,p_1}(\R^d)\bigr) =
  X^{s_0,p_0}_E(\R^d) \cap X^{s_1,p_1}_E(\R^d)$ is dense in
  $X^{s_i,p_i}_E(\R^d)$. This then immediately transfers to density of
  $X^{s_0,p_0}_E(\Omega) \cap X^{s_1,p_1}_E(\Omega)$ in
  $X^{s_i,p_i}_E(\Omega)$.

  Moreover, the spaces $X^{s_i,p_i}_E(\Omega)$ are reflexive: They are
  factor spaces of $X^{s_i,p_i}_E(\R^d)$ which are reflexive because
  they are complemented subspaces of the reflexive spaces
  $X^{s_i,p_i}(\R^d)$ as already seen above.
\end{proof}

\subsection{Operators}\label{sec:operators}

Finally, let us define the elliptic operators in divergence form and
associated operators. We first establish the usual intrinsic norm on
$W^{1,p}_D(\Omega)$, which so far only carries the abstract quotient
norm inherited from $W^{1,p}_D(\R^d)$. For $E \subset \R^d$, let us
define
\begin{equation*}
  C_E^\infty(\R^d) \defn \Bigl\{f \in C_c^\infty(\R^d) \colon
  \dist(\supp f,E) > 0\Bigr\}, \quad \text{and} \quad
  C_E^\infty(\Omega) \defn C_E^\infty(\R^d)_{\restriction \Omega}.
\end{equation*}

\begin{lemma}[{\cite[Prop.~B.3]{BM19}}]
  \label{lem:sobolev-space-intrinsic}
  Let $p \in (1,\infty)$. Then
  \begin{equation*}
    \|f\|_{W^{1,p}(\Omega)}^* \defn \left(\|f\|_{L^p(\Omega)}^p +
      \|\nabla f\|_{L^p(\Omega)}^p\right)^{\frac1p}
  \end{equation*}
  is an equivalent, intrinsic norm on $W^{1,p}_D(\Omega)$. In fact,
  $W^{1,p}_D(\Omega)$ is the closure of $C_D^\infty(\Omega)$ in this
  norm.
\end{lemma}

\begin{definition}[Coefficient functions]
  \label{def:coefficient-functions}
  Let $0< \rho_\bullet \leq \rho^\bullet$ be given. We define
  $\cC(\rho_\bullet,\rho^\bullet)$ to be the set of all measurable functions
  $\rho \colon \Omega \to \C^{d\times d}$ such that
  
  $\Re \xi^H\rho(x)\xi \geq \rho_\bullet \|\xi\|_2$ and
  $\|\rho(x)\|_{\cL(\C^{d} \to \C^{d})} \leq \rho^\bullet$ hold
  true for almost all $x \in \Omega$ and all $\xi \in \C^d$.
\end{definition}

From now on, whenever we refer to  $\cC(a,b)$ we
tacitly assume  $0< a \leq b$.

\begin{definition}[Second-order elliptic operator in divergence form]
  \label{def:div-grad}
  Let $\rho \in \cC(\rho_\bullet,\rho^\bullet)$. We define the second-order operator
  $-\nabla \cdot \rho \nabla$ by
  \begin{equation*}
    \bigl\langle -\nabla\cdot\rho\nabla u,v\bigr\rangle \defn
    \int_\Omega \rho\nabla u \cdot \overline{\nabla v}.
  \end{equation*}
  By the assumption on $\rho$, it is clear that
  $-\nabla \cdot\rho\nabla \in \cL(W^{1,p}_D(\Omega) \to
  W^{-1,p}_D(\Omega))$ for all $p \in (1,\infty)$, with the
  operator norm bounded by $\rho^\bullet$.
\end{definition}

\begin{rem}
  \label{rem:div-grad-norm}
  \begin{enumerate}[(i)]
  \item For $p=2$, based on
    \cref{lem:sobolev-space-intrinsic}, the Lax-Milgram
    lemma implies that $-\nabla\cdot\rho\nabla$ is continuously
    invertible whenever $\mathbf{1} \notin W^{1,2}_D(\Omega)$,
    and in this case the norm of the inverse is bounded by
    $\rho_\bullet^{-1}$.
  \item In connection with the previous point and the
    introduction with the elliptic regularity
    property~\eqref{eq:intro-assu}, let us point out that
    $-\nabla\cdot\rho\nabla$ will in general \emph{not} be
    surjective as an operator
    $W^{1,p}_D(\Omega) \to W^{-1,p}_D(\Omega)$ for $p\neq2$,
    even if $\mathbf{1} \notin W^{1,p}_D(\Omega)$. This is why
    often the maximal co-restriction to, say,
    $W^{-1,p}_D(\Omega)$ for $p>2$ of
    $-\nabla\cdot\rho\nabla \colon W^{1,2}_D(\Omega) \to
    W^{-1,2}_D(\Omega)$ is considered, as an unbounded operator
    in $W^{-1,p}_D(\Omega)$. We will however not need this
    distinction for this work.
  \end{enumerate}

\end{rem}

\begin{definition}[First-order operators]
  \label{def:first-order-term}
  Let $\beta \in L^\infty(\Omega;\C^d)$. We define the first-order
  operators $-\nabla\cdot \beta$ and $\beta \cdot \nabla$ by
  \begin{align*}
    \blangle -\nabla\cdot \beta u,v\brangle \defn \int_\Omega u \,
    \beta \cdot \overline{\nabla v} \quad \text{and} \quad  \blangle \beta \cdot \nabla u,v\brangle \defn \int_\Omega \beta\cdot\nabla u 
    \, \overline{v}.
  \end{align*}
  The operators give rise to
  continuous linear operators
  $W^{1,p}_D(\Omega) \to W^{-1,p}_D(\Omega)$ for every
  $p \in (1,\infty)$. This follows via Sobolev embedding.
\end{definition}

 We next introduce a suitable trace operator for functions
in $W^{s,p}(\Omega)$.

\begin{lemma}[{\cite[Thm.~8.7 (iii)]{BMMM}}]
  \label{lem:slobo-domain-trace-coincide}
  Let $p \in (1,\infty)$ and $s \in (\frac1p,1+\frac1p)$. Let
  $E \subseteq \overline\Omega$ be a $(d-1)$-set and consider
  $u \in W^{s,p}(\Omega)$. Then the \emph{inner trace} $\itr_E u$ given
  by
  \begin{equation*}
    (\itr_E u)(x) := \lim_{r \searrow 0} \frac {1}{|B_r(x) \cap
      \Omega|} \int_{{B_r(x) \cap \Omega}} u  \quad (x \in E)
  \end{equation*}
  is well defined and coincides with the trace of any
  $W^{s,p}(\R^d)$-extension of $u$, that is,
  $\itr_E u = \tr_E \widehat u$ for all $\widehat u \in W^{s,p}(\R^d)$
  such that $\widehat u_{\restriction \Omega} = u$.
\end{lemma}

We refer to \cref{rem:ahlfors-thickness} regarding the
assumption \emph{$d$-thick} in~\cite{BMMM}. In view of the foregoing
\cref{lem:slobo-domain-trace-coincide}, there will be no
ambiguity if we use the notation $\tr_E$ also for the interior trace
operator on $W^{s,p}(\Omega)$. We thus do so from now on.

\begin{corollary}
  \label{cor:slobo-domain-trace-op}
  Let $p \in (1,\infty)$ and $s > \frac1p$. Let
  $E \subseteq \overline\Omega$ be a $(d-1)$-set.
  \begin{enumerate}
  \item Let $sp < d$ and $s + \frac{d-1}q = \frac{d}p$. Then
    $\tr_E \colon W^{s,p}(\Omega) \to L^r(E;\cH_{d-1})$ is continuous
    for $r=q$ and even compact for $r \in [1,q)$.
  \item Let $sp > d$. Then
    $\tr_E \colon W^{s,p}(\Omega) \to L^\infty(E;\cH_{d-1})$ is compact.
  \end{enumerate}
\end{corollary}

\begin{proof}
  There is a continuous extension operator
  $W^{s,p}(\Omega) \to W^{s,p}(\R^d)$ by~\cite[Thm.~VI.1]{JW84} since
  $\Omega$ is a $d$-set by assumption; cf.\ also
  \cref{rem:intrinsic-slobodetskii}.  It is sufficient to
  establish the claims for $s \in (\frac1p,1+\frac1p)$ due to Sobolev
  embedding. Thus, we can rely on
  \cref{lem:slobo-domain-trace-coincide} to derive the desired
  properties from the trace operator on the full space in this case.
  \begin{enumerate}[(i)] 
  \item It is sufficient to establish the continuity assertion
    for $r=q$. To this end, we combine~\cite[Thm.~6.8]{Bie09}
    with~\cite[Thm.~V.1]{JW84} applied to $E$. This
    shows that $\tr_E \colon W^{s,p}(\R^d) \to L^q(E;\cH_{d-1})$ is
    continuous.  Regarding compactness, let us note that if
    $r \in [1,q)$, then $s + \frac{d-1}r > \frac{d}p$, hence
    $(\frac{d-1}r-\frac{d}p,s) \neq \emptyset$. Choosing
    $\alpha$ from that interval, we have
    $W^{s,p}(\R^d) \embeds H^{\alpha,p}(\R^d)$ by classical
    embeddings. Now the proof of~\cite[Cor.~7.3]{Bie09} applies
    \emph{mutatis mutandis}.
  \item In this case, every function from $W^{s,p}(\R^d)$ admits a
    H\"older continuous bound\-ed representative by classical Sobolev
    embedding. The assertion follows from the Arzel\`{a}-Ascoli
    theorem. 
  \end{enumerate}
\end{proof}

With the foregoing \cref{cor:slobo-domain-trace-op}, the
following is well defined:

\begin{definition}
  \label{def:boundary-operator}
  Let $E \subseteq \overline\Omega$ be a $(d-1)$-set and let $\varrho
  \in L^\infty(E;\cH_{d-1})$. We define
  \begin{equation*}
    \blangle \tr_E^* \varrho \tr_E
    u,v\brangle \defn \int_E \varrho \, (\tr_E u) \,
    \overline{(\tr_E v)} \dd \cH_{d-1}.
  \end{equation*}
  The operators $\tr_E^* \varrho \tr_E$ define continuous
  linear operators $W^{1,p}_D(\Omega) \to W^{-1,p}_D(\Omega)$
  for every $p \in (1,\infty)$.
\end{definition}

We next put all the above defined operators to work for our main result.

\section{Extrapolation of elliptic regularity}\label{sec:extr-ellipt-regul}

In this section, we establish the main result,
\cref{thm:invert-extrapol}.  We first quote the Sneiberg
theorem in a quantitative version from~\cite[Appendix]{AuschEg}. It is
the abstract result which will allow us to extrapolate the isomorphism
property.

\begin{theorem}[Quantitative Sneiberg]
  \label{thm:sneiberg}
  Let $(X_0, X_1)$ and $(Y_0,Y_1)$ be interpolation couples of Banach
  spaces, and let $A$ be a continuous linear operator satisfying
  $A \in \cL(X_0\to Y_0) \cap \cL(X_1 \to Y_1)$. Then the set
  \begin{equation*}
    \cI(A) \defn\Bigl\{\theta \in (0,1) \colon A \in
    \cLiso\bigl([X_0,Y_0]_{\theta}\to[X_1,Y_1]_{\theta}\bigr)\Bigr
    \} 
  \end{equation*}
  is an open interval. In fact, suppose that $\bar\theta \in \cI(A)$
  and consider $\kappa > 0$ such that
  \begin{equation*}
    \|Ax\|_{[X_1,Y_1]_{\bar\theta}} \geq \kappa \|x\|_{[X_0,Y_0]_{\bar\theta}}
    \quad \text{for all}~x\in [X_0,Y_0]_{\bar\theta}.
  \end{equation*}
  Then \begin{equation}
    \label{eq:sneiberg-interval}
    \bigl|\theta-\bar\theta\bigr| \leq
    \frac{\kappa\max\bigl(\bar\theta,1-\bar\theta\bigr)}{6\kappa + 12
      \max\bigl(\|A\|_{ \cL(X_0;Y_0)},\|A\|_{\cL(X_1;Y_1)}\bigr)}
  \end{equation}
  implies that $\theta \in \cI(A)$ with
  $\|A^{-1}\|_{[X_1,Y_1]_\theta\to[X_0,Y_0]_\theta} \leq
  8\kappa^{-1}$. 
\end{theorem}

Of course, $\cI(A)$ in \cref{thm:sneiberg} can be empty. Since
the Slobotedskii scale is obtained by \emph{real} interpolation,
see~\eqref{eq:real-interpolation}, we also give the following
corollary to \cref{thm:sneiberg} considering the real
interpolation scale.

\begin{corollary} \label{cor:real-invert} Adopt the setting of
  \cref{thm:sneiberg}. Then
  \begin{equation*}
    \cI(A) \subseteq \Bigl\{\theta \in (0,1) \colon A \in
    \cLiso\bigl((X_0,Y_0)_{\theta,q}\to(X_1,Y_1)_{\theta,q}\bigr)\Bigr
    \}
  \end{equation*}
  for all $q \in [1,\infty]$.
\end{corollary}

\begin{proof}
  Let $\theta \in \cI(A)$. Since $\cI(A)$ is open by
  \cref{thm:sneiberg}, we can choose $\tau,\sigma \in \cI(A)$
  and $\lambda \in (0,1)$ such that
  $\theta=(1-\lambda) \tau + \lambda \sigma$. Then
  \begin{equation*} A \colon
    \bigl([X_0,X_1]_\tau,[X_0,X_1]_\sigma\bigr)_{\lambda,q} \to
    \bigl([Y_0,Y_1]_\tau,[Y_0,Y_1]_\sigma\bigr)_{\lambda,q}
  \end{equation*}
  remains continuously invertible for all $q \in [1,\infty]$. But, by
  re-iteration, the space on the left hand side is
  $(X_0,X_1)_{(1-\lambda) \tau + \lambda \sigma,q} =
  (X_0,X_1)_{\theta,q}$, and the one on the right hand side is
  $(Y_0,Y_1)_{(1-\lambda) \tau + \lambda \sigma,q} =
  (Y_0,Y_1)_{\theta,q}$, cf.~\cite[Thm.~1.10.3.2]{Triebel}.
\end{proof}

Our next intermediate goal is to extend the gradient
$\nabla \colon H^{1,p}(\Omega) \to L^p(\Omega)^d$ continuously to a
mapping $H^{1-s,p}(\Omega) \to H^{-s,p}(\Omega)^d$. This will then
allow to also extend the elliptic operator $-\nabla\cdot\rho\nabla$,
cf.\ \cref{lem:div-grad-extend} below. To this end, we
first quote the following result regarding continuity of the zero
extension in the low regularity regime. (See \cref{r-d-1} to
validate its assumptions.)

\begin{lemma}[{\cite[Cor.~2.18]{BM19}}]
  \label{lem:zero-extension}
  Let $p \in (1,\infty)$ and $s \in [0,\frac1p)$. Then the zero
  extension
  \begin{equation*}
    (\cE_0f)(x) =
    \begin{cases}
      f(x) & \text{if}~x \in \Omega, \\ 0 & \text{otherwise}
    \end{cases}
  \end{equation*}
  is a continuous linear operator
  $\cE_0 \colon X^{s,p}(\Omega) \to X^{s,p}(\R^d)$ for both $X=H$ or
  $W$.
\end{lemma}

\begin{lemma}
  \label{lem:density-ccinfty-low-reg}
  Let $p \in (1,\infty)$ and $s \in [0,\frac1p)$. Then
  $C_{\partial\Omega}^\infty(\Omega)$ is dense in $H^{s,p}(\Omega)$.
\end{lemma}

\begin{proof}
  It is enough to show that $H^{s,p}(\Omega)$ is a subset of the
  closure $H^{s,p}_0(\Omega)$ of $C_{\partial\Omega}^\infty(\Omega)$ in the
  $H^{s,p}(\Omega)$ norm. Let $f \in
  H^{s,p}(\Omega)$. \cref{lem:zero-extension} asserts that
  $\cE_0f \in H^{s,p}(\R^d)$. Clearly, $\cE_0f = 0$ on
  $\R^d\setminus\Omega$. A theorem of
  Netrusov~(\cite[Thm.~10.1.1]{Adams_Hedberg}) thus implies that
  $f\in H^{s,p}_0(\Omega)$.
\end{proof}

\begin{lemma}
  \label{lem:gradient-extension}
  Let $p \in (1,\infty)$ and $s \in (0,\frac1p \wedge
  \frac1{p'})$. Then the weak gradient
  $\nabla \in \cL(H^{1,p}(\Omega)\to L^p(\Omega)^d)$ maps
  $H^{1+s,p}(\Omega)$ continuously nonexpansively into
  $H^{s,p}(\Omega)^d$ and admits a unique continuous linear and still
  nonexpansive extension to a mapping
  $\nabla \colon H^{1-s,p}(\Omega) \to H^{-s,p}(\Omega)^d$.
\end{lemma}

\begin{proof}
  The proof is based on the observation that the distributional
  (partial) derivative $\partial_j$, $j \in \{1,\dots,d\}$, is a
  continuous linear contraction from $H^{\sigma,q}(\R^d)$ to
  $H^{\sigma-1,q}(\R^d)$ for all $\sigma \in \R$ and all
  $q \in (1,\infty)$. This in turn can be seen e.g.\ for $\sigma$ an
  integer via $H^{k,q}(\R^d) = W^{k,q}(\R^d)$ for $k \in \N_0$ and a
  duality argument; the general case for $\sigma$ then follows by
  interpolation. Moreover, this distributional derivative is of course
  consistent with the weak derivative on $H^{1,q}(\R^d)$.
  
  The first claim thus follows immediately from the definitions of
  $H^{1+s,p}(\Omega)$ and $H^{s,p}(\Omega)$ as the restrictions of the
  corresponding spaces on $\R^d$. For the second one, consider
  $f \in H^{1-s,p}(\Omega)$ and let $\widehat f \in H^{1-s,p}(\R^d)$
  be such that $\widehat{f}_{\restriction \Omega} = f$. Let moreover
  $\varphi \in C_{\partial\Omega}^\infty(\Omega)$ and identify it with its extension
  by zero $\cE_0 \varphi$ to $\R^d$. Then
  $\cE_0\varphi \in H^{s,p'}(\R^d)$ by \cref{lem:zero-extension}
  and in fact
  $\|\varphi\|_{H^{s,p'}(\Omega)} =
  \|\cE_0\varphi\|_{H^{s,p'}(\R^d)}$. Let $j \in \{1,\dots,d\}$. We
  observe that
  \begin{equation*}
    \bigl\langle \partial_j f,\varphi
    \bigr\rangle \defn -\int_\Omega f \, \partial_j \varphi = -\int_{\R^d} \widehat f \,
    \partial_j \cE_0\varphi,
  \end{equation*}
  hence
  \begin{multline*}
    \left\lvert\bigl\langle \partial_j f,\varphi
      \bigr\rangle\right\rvert \leq \bigl\|\widehat
    f\bigr\|_{H^{1-s,p}(\R^d)}\bigl\|\partial_j
    \cE_0\varphi\|_{H^{s-1,p'}(\R^d)} \\ \leq \bigl\|\widehat
    f\bigr\|_{H^{1-s,p}(\R^d)}
    \bigl\|\cE_0\varphi\bigr\|_{H^{s,p'}(\R^d)} = \bigl\|\widehat
    f\bigr\|_{H^{1-s,p}(\R^d)}
    \bigl\|\varphi\bigr\|_{H^{s,p'}(\Omega)}.
  \end{multline*}
  Note that $C_{\partial\Omega}^\infty(\Omega)$ is dense in $H^{s,p'}(\Omega)$ since
  $s \in [0,1-\frac1p)$, cf.\
  \cref{lem:density-ccinfty-low-reg}. Thus, taking the infimum
  over all $\widehat f \in H^{1-s,p}(\R^d)$ such that
  $\widehat f_{\restriction \Omega} = f$, we find
  $\partial_j \in \cL(H^{1-s,p}(\Omega)\to H^{-s,p}(\Omega))$, since
  $H^{-s,p}(\Omega) = (H^{s,p'}(\Omega))^\star$ by definition.
\end{proof}

We also need the notion of a \emph{multiplier}.

\begin{definition}[Multiplier]
  \label{def:multiplier}    
  Let $X$ be a Banach space of functions $\Omega \to \C$.
  \begin{enumerate}[(i)]
  \item A function $\omega \colon \Omega \to \C$ is a \emph{multiplier
      on $X$} if the superposition operator $M_\omega$ defined by
    $(M_\omega f)(x) \defn \omega(x)f(x)$ maps $X$ continuously into
    itself. We write $\omega \in
    \cM(X)$ and the multiplier norm is given by $\|\omega\|_{\cM(X)} \defn \|M_\omega\|_{X\to X}$.
  \item For a matrix function
    $\omega \colon \Omega \to \C^{d\times d}$ where each component
    satisfies $\omega_{ij} \in \cM(X)$, we use the associated
    \emph{multiplier norm} defined by
  \begin{equation*}
    \|\omega\|_{\cM(X)}=\sqrt{\sum_{i=1}^m \sum_{j=1}^n \|\omega_{ij}\|_{\cM(X)}^2}.
  \end{equation*}
  \end{enumerate}
\end{definition}

Using multiplier assumptions, all of the differential and boundary
operators introduced in \cref{sec:operators} can be extended to
the Bessel scale. The collected result is as follows:

\begin{lemma}
  \label{lem:div-grad-extend}
  Let $p \in (1,\infty)$ and $\tau\in (0,\frac1p \wedge
  \frac1{p'})$, and let moreover the following assumptions be satisfied:
  \begin{itemize}
  \item $\rho\colon\Omega \to \C^{d\times d}$ such that
    $\rho_{ij} \in \cM(H^{\tau,p}(\Omega)) \cap \cM(H^{\tau,p'}(\Omega))$,
  \item $\beta_{\dive},\beta_g \in \cM(H^{\tau,p}(\Omega))^d
    \cap \cM(H^{\tau,p'}(\Omega))^d$, 
  \item $\eta \in L^d(\Omega)$,
  \item $E \subseteq \overline\Omega$ is a $(d-1)$-set and
    $\varrho \in L^\infty(E;\cH_{d-1})$.
  \end{itemize}
  Then the operator $A$ defined by
  \begin{equation}
    A \defn -\nabla\cdot\rho\nabla - \nabla \cdot \beta_{\dive} + \beta_g
    \cdot \nabla + \eta + \tr_E^*\varrho\tr_E\label{eq:diff-operator}
  \end{equation} maps
  $H^{1+\tau,p}_D(\Omega)$ continuously into
  $H^{\tau-1,p}_D(\Omega)$, and linearly extends to a continuous
  mapping from $H^{1-\tau,p}_D(\Omega)$ to $H^{-1-\tau,p}_D(\Omega)$.
\end{lemma}

\begin{proof} We first show that $-\nabla\cdot\rho\nabla$ maps
  $H^{1+\tau,p}_D(\Omega)$ continuously into $H^{\tau-1,p}_D(\Omega)$
  using the multiplier assumption. So, let
  $\varphi \in H^{1+\tau,p}_D(\Omega)$ and
  $\psi \in W^{1,p'}_D(\Omega)$. Then
  $\nabla\psi \in L^{p'}(\Omega) \subset H^{-\tau,p}_D(\Omega)$, and
  using \cref{lem:gradient-extension}, we find
  \begin{align*}
    \bigl\langle-\nabla\cdot\rho\nabla
      \varphi,\psi\bigr\rangle =
    \bigl(\rho\nabla\varphi,\nabla\psi\bigr)_{L^2(\Omega)^d}
    & \leq \|\rho\|_{\cM(H^{\tau,p}(\Omega))}
    \|\nabla\varphi\|_{H^{\tau,p}(\Omega)^d}\|\nabla\psi\|_{H^{-\tau,p'}(\Omega)^d}
    \\& \leq \|\rho\|_{\cM(H^{\tau,p}(\Omega))}
    \|\varphi\|_{H^{1+\tau,p}_D(\Omega)}\|\psi\|_{H^{1-\tau,p'}_D(\Omega)}.
  \end{align*}
  Since $W^{1,p'}_D(\R^d)$ is dense in $H^{1-\tau,p'}_D(\R^d)$, so is
  $W^{1,p'}_D(\Omega)$ in $H^{1-\tau,p'}_D(\Omega)$. Accordingly,
  $-\nabla\cdot\rho\nabla$ maps $H^{1+\tau,p}_D(\Omega)$ continuously
  into $H^{\tau-1,p}_D(\Omega)$.

  Next, we show that $-\nabla\cdot\rho\nabla$ continuously extends to
  an operator from $H^{1-\tau,p}_D(\Omega)$ to
  $H^{-1-\tau,p}_D(\Omega)$. We follow the same reasoning as above,
  this time for $\varphi \in W^{1,p}_D(\Omega)$ and
  $\psi \in H^{1+\tau,p'}_D(\Omega)$, to obtain
  \begin{equation*}
    \bigl\langle-\nabla\cdot\rho\nabla
      \varphi,\psi\bigr\rangle =
   \bigl(\nabla\varphi,\rho^H\nabla\psi\bigr)_{L^2(\Omega)^d}
    \leq \|\rho^H\|_{\cM(H^{\tau,p'}(\Omega))}
    \|\varphi\|_{H^{1-\tau,p}_D(\Omega)}\|\psi\|_{H^{1+\tau,p'}_D(\Omega)}.
  \end{equation*}
  Density of $W^{1,p}_D(\Omega)$ in $H^{1-\tau,p}_D(\Omega)$ then
  yields that $-\nabla\cdot\rho\nabla$ extends continuously to
  $H^{1-\tau,p}_D(\Omega)$, mapping into $H^{-1-\tau,p}_D(\Omega)$.

  The first-order operators $\nabla\cdot\beta_{\dive}$ and
  $\beta_g \cdot \nabla$ work exactly analogously. For the zero-order
  operator, the claim follows from Sobolev embeddings and H\"older's
  inequality. Let us thus turn to the boundary form operator. Choose
  $s \in (\tau,\frac1p \wedge \frac1{p'})$.
  Letting $u \in H^{1+\tau,p}_D(\Omega)$ and
  $v \in H^{1+\tau,p'}_D(\Omega)$, we estimate easily via
  \cref{cor:slobo-domain-trace-op}:
  \begin{align}
    \blangle \tr_E^*\varrho\tr_E u,v\brangle  &\leq
    \|\varrho\|_{L^\infty(E;\cH_{d-1})}
    \|\tr_E u\|_{L^p(E;\cH_{d-1})}\|\tr_E v\|_{L^{p'}(E;\cH_{d-1})} \notag \\ &
    \lesssim
    \|\varrho\|_{L^\infty(E;\cH_{d-1})}\|u\|_{W^{1-s,p}(\Omega)}
    \|v\|_{W^{1-s,p'}(\Omega)}. \label{eq:boundary-op-estimate}   
  \end{align}
  Now the assertion follows from the embeddings~(\cite[Thm.~4.6.1]{Triebel})
  \begin{equation*}
    H^{1+\tau,p}_D(\Omega) \embeds H^{1-\tau,p}_D(\Omega) \embeds
    W^{1-s,p}(\Omega), \quad 
    H^{1+\tau,p'}_D(\Omega) \embeds H^{1-\tau,p'}_D(\Omega) \embeds W^{1-s,p'}(\Omega),
  \end{equation*}
  where the first ones in the respective chain are dense.
\end{proof}

\begin{rem}
  \label{rem:multipliers} Let $p \in (1,\infty)$ and $\tau \in (0,\frac1p)$.
  \begin{enumerate}[(i)]
  \item All multipliers considered will be bounded: 
    $\cM(H^{\tau,p}(\Omega)) \embeds \cM(L^p(\Omega))$ and $L^\infty(\Omega) =
    \cM(L^p(\Omega))$, the latter up to equivalent norms. Indeed, note
    that the constant function $\mathbf{1}$ is an element of
    $H^{\tau,p}(\Omega)$. So let $k \in \N$ and consider for
    $\omega \in \cM(H^{\tau,p}(\Omega))$:
  \begin{equation*}
    \|\omega
    \|_{L^{pk}(\Omega)} = \|\omega^k
    \mathbf{1}\|_{L^p(\Omega)}^{\frac1k} \lesssim    \|\omega^k 
    \mathbf{1}\|_{H^{\tau,p}(\Omega)}^{\frac1k} \lesssim
    \|\omega\|_{\cM(H^{\tau,p}(\Omega))}
    \|\mathbf{1}\|_{H^{\tau,p}(\Omega)}^{\frac1k}. 
  \end{equation*} 
  Since $\|\mathbf{1}\|_{H^{\tau,p}(\Omega)}^{\frac1k} \lesssim 1$, it
  follows by contradiction that $\omega \in L^\infty(\Omega)$, and
  taking the limit as $k \to \infty$ gives the desired embedding. It
  is easy to see that $L^\infty(\Omega)$ and $\cM(L^p(\Omega))$ are
  isomorphic. Note moreover that $\cM(H^{\tau,p}(\Omega)) \embeds \cM(L^p(\Omega))$
  implies that
  $\cM(H^{\tau,p}(\Omega)) \embeds \cM(H^{\sigma,p}(\Omega))$ for all
  $\sigma \in [0,\tau]$ via complex
  interpolation~(\cite[Rem.~3.9]{BM19}).
\item We do not have a general description of
  $\cM(H^{\tau,p}(\Omega))$ for $\tau > 0$ in terms of classical
  function spaces. However, there is a substantial body of work
  devoted to multipliers on the usual function spaces; we mention
  exemplarily the comprehensive books~\cite{MazyaMult,RunstSickel},
  or~\cite[Sect.~5]{Marschall}. We give a few examples. Most
  generally, due to the condition $\tau < \frac1p$,
  \cref{lem:zero-extension} implies that
  $\cM(H^{\tau,p}(\R^d)) \embeds \cM(H^{\tau,p}(\Omega))$. It is
  moreover a classical result that
  $C^\sigma(\Omega) \embeds \cM(H^{\tau,p}(\Omega))$ whenever
  $\tau< \sigma \leq 1$, where $C^\sigma(\Omega)$ denotes the space of
  $\sigma$-H\"older continuous functions. In fact, already a slightly
  larger Besov space does the job:
  $B^{\tau}_{\infty,p}(\Omega) \embeds \cM(H^{\tau,p}(\Omega))$. We
  refer to e.g.~\cite[Ch.~4.7.1]{RunstSickel}
  and~\cite[Lem.~1]{haller}, where it is also mentioned that
  $B^\tau_{\infty,p}(\Omega) \embeds C^\tau(\Omega)$. But continuity
  is not at all necessary for the multiplier property, in particular
  in the present low-regularity case of $\tau < \frac1p$: already the
  characteristic functions $\chi_\Lambda$ of certain subsets
  $\Lambda \subset \Omega$ are also multipliers on
  $H^{\tau,p}(\Omega)$. Examples for sets $\Lambda$ with this property
  are convex sets~(\cite[Rem.~3.5.3]{MazyaMult}) or sets of locally
  finite perimeter~(\cite[p.~214ff]{RunstSickel}); see
  also~\cite{sickel} for the probably most general admissible
  class. (In fact,~\cite[Thm.~4.4]{sickel} provides
  \cref{lem:zero-extension}.)
  \end{enumerate}
\end{rem}

The following is our main result for this section. It holds
for both $X \in \{H,W\}$.

\begin{theorem}
  \label{thm:invert-extrapol}
  Let $p \in (1,\infty)$ and
  $\tau\in (0,\frac1p \wedge \frac{1}{p'})$. Let $A$ be as
  in~\eqref{eq:diff-operator} and let the following
  assumptions on the data be satisfied, as in \cref{lem:div-grad-extend}:
    \begin{itemize}
  \item $\rho\colon\Omega \to \C^{d\times d}$ such that
    $\rho_{ij} \in \cM(H^{\tau,p}(\Omega)) \cap \cM(H^{\tau,p'}(\Omega))$,
  \item $\beta_{\dive},\beta_g \in \cM(H^{\tau,p}(\Omega))^d
    \cap \cM(H^{\tau,p'}(\Omega))^d$, 
  \item $\eta \in L^d(\Omega)$,
  \item $E \subseteq \overline\Omega$ is a $(d-1)$-set and
    $\varrho \in L^\infty(E;\cH_{d-1})$.
  \end{itemize}
  Suppose further that there is $\lambda \in \C$ such that
  \begin{equation*}
  A + \lambda  \in \cLiso\bigl(W^{1,p}_D(\Omega)\to W^{-1,p}_D(\Omega)\bigr).
\end{equation*}
  Then there is
  $\bar s \in (0,\tau]$ such that
  \begin{equation} \label{e-extiso} A + \lambda \in
    \cLiso\bigl(X^{1+s,p}_D(\Omega) \to X^{s-1,p}_D(\Omega)\bigr) \quad
    (s\in (-\bar s,\bar s)).
  \end{equation}
  Further, both $\bar s$ and the norms of the inverse operators $(A + \lambda)^{-1}$
  between $X^{s-1,p}_D(\Omega)$ and $X^{1+s,p}_D(\Omega)$ for
  $s \in (-\bar s,\bar s)$ can be estimated uniformly in the norm of
  all the given data and
  $ \|A + \lambda\|^{-1}_{W^{1,p}_D(\Omega) \to W^{-1,p}_D(\Omega)}$.
\end{theorem}

\begin{proof}
  We only need to collect several results from above and combine them
  with the Sneiberg \cref{thm:sneiberg}. First, due to
  \cref{lem:div-grad-extend}, we already know that $A$ gives rise
  to continuous linear operators
  $H^{1+\tau,p}_D(\Omega) \to H^{\tau-1}_D(\Omega)$ and
  $H^{1-\tau,p}_D(\Omega) \to H^{-1-\tau,p}_D(\Omega)$, and it is
  clear that this extends to $A +\lambda$.
  
  Second, we note that, by
  \cref{prop:interpolation-zero-trace} and
  \cref{cor:interpolation-dual},
  \begin{equation*}
    W^{1,p}_D(\Omega) =
    \bigl[H^{1+\tau,p}_D(\Omega),H^{1-\tau,p}_D(\Omega)\bigr]_{\frac12},
    \quad 
    W^{-1,p}_D(\Omega) =
    \bigl[H^{\tau-1,p}_D(\Omega),H^{-1-\tau,p}_D(\Omega)\bigr]_{\frac12}.
  \end{equation*}
  From \cref{cor:real-invert} and \cref{thm:sneiberg}
  we thus infer that there is $\eps \in (0,\frac12]$ such that the
  operators
  \begin{equation*}A  + \lambda\colon
    \begin{cases} \quad
      \bigl[H^{1+\tau,p}_D(\Omega),H^{1-\tau,p}_D(\Omega)\bigr]_{\delta}
      &\to \quad
      \bigl[H^{\tau-1,p}_D(\Omega),H^{-1-\tau,p}_D(\Omega)\bigr]_{\delta}
      \\[1em] \quad
      \bigl(H^{1+\tau,p}_D(\Omega),H^{1-\tau,p}_D(\Omega)\bigr)_{\delta,p}
      &\to \quad \bigl(H^{\tau-1,p}_D(\Omega),H^{-1-\tau,p}_D(\Omega)\bigr)_{\delta,p}
    \end{cases}    
  \end{equation*}
  remain continuously invertible for all $\delta \in (\frac12-\eps,\frac12+\eps)$.
  But according to \cref{prop:interpolation-zero-trace},
  the former spaces coincide with
  $H^{1+s,p}_D(\Omega) \to H^{1-s,p}_D(\Omega)$ and the latter ones
  with $W^{1+s,p}_D(\Omega) \to W^{1-s,p}_D(\Omega)$, where we have
  set $s \defn \tau(1-2\delta)$. The range of $\delta$ then
  corresponds to $s \in (-\bar s,\bar s)$ where
  $\bar s \defn 2\tau\eps$. Thus we obtain exactly~\eqref{e-extiso}.

  The claimed uniformity of $\bar s$ and the norms of the inverses of
  $A + \lambda$ follows immediately from~\eqref{eq:sneiberg-interval} in
  \cref{thm:sneiberg} and the associated norm estimate,
  together with the estimates on the extension and restriction of $A$
  to the Bessel scale as obtained in \cref{lem:div-grad-extend}.
\end{proof}

Note that~\cite{disser} gives a comprehensive list of settings where
the principal part $-\nabla\cdot\rho\nabla$ (or
$-\nabla\cdot\rho\nabla + \lambda$) of $A$ alone satisfies the
isomorphism assumption in \cref{thm:invert-extrapol}. It thus seems
appropriate to state an auxiliary result leading to the corresponding
assumption for $A$, starting from just the principal part.

\begin{corollary} \label{cor:rand-inv} Let $p \geq 2$. Let
  $\lambda \in \C$ and suppose the following on the data:
 \begin{itemize}
  \item $\rho \in \cC(\rho_\bullet,\rho^\bullet)$,
  \item $\beta_{\dive} \in L^\infty(\Omega))^d$ and there is $\tau \in
    (0,\frac1p)$ such that $\beta_g \in \cM(H^{\tau,p'}_D(\Omega))^d$, 
  \item $\eta \in L^d(\Omega)$ and there exists $\eta_\bullet \in
    \R$ such that $\Re \eta \geq \eta_\bullet$ a.e.\ on $\Omega$,
  \item $E \subseteq \overline\Omega$ is a $(d-1)$-set,
    $\varrho \in L^\infty(E;\cH_{d-1})$ and $\Re\varrho
    \geq \varrho_\bullet \geq 0$ in the $\cH_{d-1}$-a.e.\ sense on $E$.
  \end{itemize}
  Moreover, assume that
  \begin{equation*}
    \alpha \defn \Re \lambda+\eta_\bullet -
    \frac{\bigl(\|\beta_{\dive}\|_{L^\infty(\Omega)}+\|\beta_g\|_{L^\infty(\Omega)}\bigr)^2}{2c_\bullet}
    \geq 0, 
  \end{equation*}
  and that $\alpha + \varrho_\bullet > 0$ if $\lambda \neq
  0$. Then
  \begin{equation*}
-\nabla \cdot \rho \nabla + \lambda \in
  \cLiso\bigl(W^{1,p}_D(\Omega)\to W^{-1,p}_D(\Omega)\bigr) \hfill \implies \hfill A+ \lambda \in
  \cLiso\bigl(W^{1,p}_D(\Omega)\to W^{-1,p}_D(\Omega)\bigr)
\end{equation*}
with $A$ as in~\eqref{eq:diff-operator}.
\end{corollary}

Recall that the assumption on $\beta_g$ implies that $\beta_g \in
L^\infty(\Omega)$, see \cref{rem:multipliers}.

\begin{proof}[Proof of \cref{cor:rand-inv}]
  We first intend to show that the lower order operator $B$, so
  $B \defn -\nabla\cdot\beta_{\dive} + \beta_g\cdot\nabla +
  \eta+\tr_E^*\varrho\tr_E$, is relatively compact on
  $W^{-1,p}_D(\Omega)$ with respect to
  $-\nabla\cdot\rho\nabla + \lambda$. Let $\tau$ be from the assumption
  on $\beta_g$. By the compactness of the embedding
  $W^{1,p}_D(\Omega) \embeds H^{1-\tau,p}_D(\Omega)$, it suffices to
  prove that $B \colon H^{1-\tau,p}_D(\Omega) \to W^{-1,p}_D(\Omega)$
  is continuous. But this is straightforward to verify from the
  assumptions; for the boundary operator $\tr_E^*\varrho\tr_E$ we
  choose $s \in (\tau,\frac1p)$ and refer to the
  estimate~\eqref{eq:boundary-op-estimate} and the embeddings
  mentioned right below.

  With $B$ relatively compact with respect to $-\nabla\cdot\rho\nabla +
  \lambda$, it follows that
  $A = -\nabla\cdot\rho\nabla +
  \lambda + B$ is (semi-)Fredholm on $W^{1,p}_D(\Omega)$ with index $0$, since
  $-\nabla\cdot\rho\nabla + \lambda$ is
  so~(\cite[Ch.~IV. Thm.~5.26]{kato}).
  Thus, it is enough to show that $A$ is injective on
  $W^{1,p}_D(\Omega)$. But this follows easily using
  $\rho \in \cC(\rho_\bullet,\rho^\bullet)$ and the conditions on $\alpha$
  and $\varrho_\bullet$. Here we also use that $p\geq2$. (Note that if
  $\lambda = 0$, then, by the isomorphism assumption,
  $\mathbf{1} \notin W^{1,p}_D(\Omega)$.)
\end{proof}

\begin{rem}
  \label{rem:strong-form}
  We complement the abstract results of
  \cref{thm:invert-extrapol} by attaching a boundary value
  problem. Let for simplicity $f \in L^p(\Omega)$ and
  $g \in L^p(\Neumann;\cH_{d-1})$ as well as $\lambda=0$. Under the
  assumptions in \cref{thm:invert-extrapol}, the abstract problem
  \begin{equation*}
    Au = f + \tr_\Neumann^* g
  \end{equation*}
  admits a unique solution $u \in H^{1+s,p}_D(\Omega)$ for some $s>0$,
  and $u$ depends continuously on $f$ and $g$. The associated boundary
  value problem is
  \begin{align*}
    -\dive\bigl(\rho\nabla u + \beta_{\dive}u\bigr) + \beta_g \cdot
    \nabla u + \eta u & = f && \text{in}~\Omega, \\
    -\rho\nabla u \cdot \nu +  \varrho u & = g && \text{on}~\Neumann, \\
    u&  = 0 && \text{on}~\Diri.
  \end{align*}
  The connection between the abstract and boundary value problem
  formulation can be made precise under additional assumptions on
  $\Omega$ which would allow to apply the divergence
  theorem; see e.g.~\cite[Ch.~1.2]{Cia} or \cite[Ch.~2.2]{GGZ}.
\end{rem}

\section{The van Roosbroeck system of semiconductor equations}\label{sec:semiconductor}

In this section we use \cref{thm:invert-extrapol} to give
a direct treatment of the van Roosbroeck system of semiconductor
equations. Here, we focus on Boltzmanns statistics only; see
however \cref{r-concl} below. The van Roosbroeck system
was already briefly introduced in the introduction and we now
give a more detailed explanation.

In the van Roosbroeck system, negative and positive charge
carriers, electrons and holes, move by diffusion and drift in a
self-consistent electrical field; on their way, they may
recombine to charge-neutral electron-hole pairs or, vice versa,
negative and positive charge carriers may be generated from
charge-neutral electron-hole pairs.  The electronic state of the
semiconductor device $\Omega \subset \R^3$ resulting from these
phenomena is described by the triple $(u_1,u_2,\varphi)$ of
unknowns consisting of the densities $u=(u_1,u_2$) of electrons
and holes and the electrostatic potential $\varphi$. Their
evolution during the (finite) time interval $J = (0,T)$ is then
described by the equations already mentioned in the
introduction, so the
\emph{Poisson equation}
\begin{equation*}
  \tag{\ref{Poisson-eq}}
  \begin{aligned}
    -\dive\left( \varepsilon \nabla \varphi \right) & = {\fd} +
    u_1 - u_2  &\quad&
    \text{in}~J\times\Omega, \\
    \varphi & = {\varphi}_\Diri &\quad&
    \text{on}~J\times \Diri, \\
    {\nu}\cdot{ \left( \varepsilon \nabla \varphi \right) } +
    \varepsilon_{\Gamma} \varphi & ={\varphi}_\Gamma &\quad&
    \text{on}~J\times\Neumann,
  \end{aligned}
\end{equation*}
and, for $k=1,2$, the \emph{current-continuity
  equations}
\begin{equation*}
  \tag{\ref{CuCo-eq}}
  \begin{aligned}
    \partial_t u_k - \dive j_k
    &= r^\Omega(u,\varphi)  &\quad& \text{in}~J\times (\Omega \setminus \Pi)\\
    u_k &=U_k &\quad&\text{on}~J\times\Diri,
    \\
    {\nu}\cdot{j_k} &= {r}^\Gamma(u,\varphi)
    &\quad&\text{on}~J\times\Neumann,
    \\
    [{\nu}\cdot{j_k}] &= {r}^\Pi(u,\varphi)&\quad&\text{on}~J\times\Pi, \\
    u_k(0) & = u_k^0 &\quad& \text{on}~\Omega,
  \end{aligned}
\end{equation*}
with the currents
\begin{equation*} \tag{\ref{eq:curr-dens}} j_k =
  \mu_k\bigl(\nabla u_k +(-1)^k u_k \nabla \varphi
  \bigr).
\end{equation*}
Let us also repeat that $\Omega \subset \R^3$ is a bounded domain,
$\nu$ its unit outer normal at $\partial\Omega$ and the latter is
decomposed into a Dirichlet part $\Diri$ and a Neumann/Robin part
$\Neumann \defn \partial\Omega\setminus \Diri$. We will require
$\Omega$ to satisfy \cref{a-geometry} and to have some
additional but in general very mild properties, specified in
\cref{sec:semiconduct-assumptions} below.

The parameters in the Poisson equation are the dielectric permittivity
$\eps\colon\Omega\to\R^{3\times 3}$ and the so-called \emph{doping
  profile} $\fd$.  The latter comes from impurities induced in the
materials or even very small layers of different, reaction-enhancing
material in the device $\Omega$, see~\cite{naz} or~\cite{drumm}. As
such we will allow it to be located only on two-dimensional surfaces
in $\overline \Omega$; see our mathematical requirement on $\fd$ in
\cref{assu:poi-diri-bv} below. Moreover, in the boundary
conditions, $\varepsilon_{\Neumann} \colon \Neumann\to[0,\infty)$
represents the capacity of the part of the corresponding device
surface, ${\varphi}_\Diri$ and ${\varphi}_\Neumann$ are the voltages
applied at the contacts of the device, thus they may depend on time. As
above, we always write $u$ for the pair of densities $(u_1,u_2)$.

Although we are aware of the fact that, from a physical point of
view, the Dirichlet data $\varphi_\Diri$ in~\eqref{Poisson-eq}
and $U_k$ in~\eqref{CuCo-eq} is---at least in case of a voltage
driven regime---an essential part of the model, we will focus on
the case where it is zero. This is in order to make the most
fundamental things in the analysis visible, for the (standard)
treatment of non-zero data see~\cite{vanroos2d}
and~\cite{vanroos3d}.

The current-continuity equations feature the
fluxes~\eqref{eq:curr-dens} with 
the mobility tensors $\mu_k \colon \Omega\to\R^{3\times 3}$ for
electrons and holes, and the recombination terms
$r^\Omega, r^\Neumann$ and $r^\Pi$. Here $r^\Omega$ models
recombination in the bulk and the normal fluxes across the exterior
boundary $\Neumann$ are balanced with surface recombination
$r^\Neumann$ taking place on $\Neumann$. For the physical significance
of interfacial recombination induced by $r^\Pi$ in modern devices we
refer to e.g.~\cite{Xiang} or~\cite[Ch.~3]{Visvana}.

The bulk recombination term $r^\Omega$ in~\eqref{CuCo-eq} can consist
of rather general functions of the electrostatic potential $\varphi$,
of the currents $j_k$, and of the vector of electron/hole densities
$u$. It describes the production, or destruction, depending on the
sign, of electrons and holes. Below, we collect some of the most
relevant examples, covering non-radiative recombination like the
Shockley-Read-Hall recombination due to phonon transition, Auger
recombination (three particle transition), and Avalanche
generation. See e.g.~\cite{selberherr84,landsberg91,gajewski93} and
the references cited there for more information. The most familiar
recombination mechanisms are the following two:
\begin{itemize}
\item \emph{Shockley-Read-Hall recombination}  (photon transition):
  \begin{equation}\label{e-shockleyrh}
    \reaction^\Omega_{\text{SRH}}(u) 
    \defn \frac{u_1 u_2 - n_i^2}{\tau_2(u_1+n_1)+\tau_1(u_2+n_2)},
  \end{equation}
  where $n_i$ is the intrinsic carrier density, $n_1$, $n_2$ are
  reference densities, and $\tau_1$, $\tau_2$ are the lifetimes of
  electrons and holes, respectively. 
\item  \emph{Auger recombination} (three particle transitions):
  \begin{equation} \label{e-auger} \reaction^\Omega_{\text{Auger}}(u) = \bigl(u_1
    u_2 - n_i^2 \bigr)\bigl(c_1^{\text{Auger}} u_1 + c_2^{\text{Auger}}u_2\bigr),
  \end{equation}
  where $c_1^{\text{Auger}}$ and $c_2^{\text{Auger}}$ are the Auger capture
  coefficients of electrons and holes, respectively, in the
  semiconductor material.
\end{itemize}
All occurring constants are parameters of the semiconductor material.

Both recombination mechanisms mentioned above depend on the carrier
densities $u$ only. This is \emph{not} the case for the Avalanche
generation term which depends also on the \emph{gradients} of the
physical quantities:
\begin{itemize}
\item An analytical expression for \emph{Avalanche generation} (impact
  ionization), valid at least in the material cases of Silicon or
  Germanium, is
  \begin{equation} \label{e-aval} r^\Omega_{\text{Ava}}(u,\varphi)=c_2 |j_2|
    \exp \Bigl ({\frac {-a_2|j_2|}{|\nabla \varphi \cdot j_2|}}\Bigr ) +c_1 |j_1| \exp \Bigl ({\frac
      {-a_1|j_1|}{|\nabla\varphi \cdot j_{1}|}}\Bigr ).
  \end{equation}
  Again, the parameters $a_1,a_2 > 0$ and $c_1,c_2$ are material-dependent. We refer
  to~\cite[p.~111/112]{selberherr84} and references; in particular
  Tables 4.2-3/4.2-4, and see also~\cite[Ch.~17, p.~54/55]{Marko}.
\end{itemize}

We give more functional-analytic meaning to the recombination terms in
the next section, where we collect the various assumptions on the data
in~\eqref{vanRoos}.

\subsection[Prerequisites]{Assumptions}
\label{sec:semiconduct-assumptions}
%

In this section, we introduce some mathematical terminology and state
mathematical prerequisites for the analysis of the van Roosbroeck
system~\eqref{vanRoos}. All assumptions in this section are supposed
to be valid from now on.

\subsubsection{Assumptions on the geometry}

We begin with the following geometric requirements on the domain
$\Omega$ occupied by the
device. \cref{fig-technoa} shows a
typical example of a semiconductor device.

\begin{figure}[htb]
  \centering
  \includegraphics[width=0.5\linewidth] {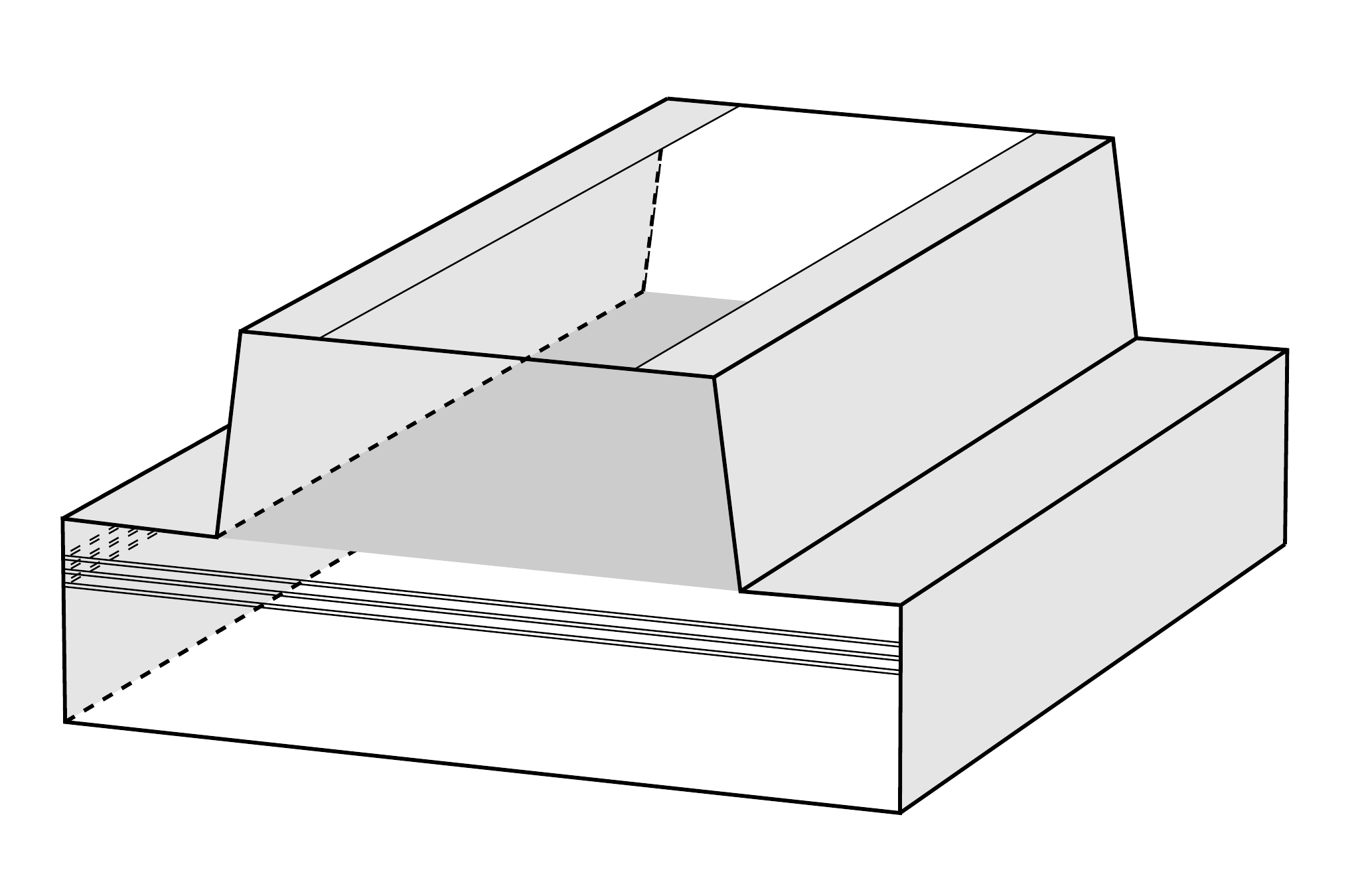}
  \caption{\label{fig-technoa} Scheme of a ridge waveguide
    quantum well laser (detail
    $3.2\mu\text{m} \times 1.5 \mu\text{m} \times
    4\mu\text{m}$). The device has two material layers, the
    material interface is the darkly shaded plane. The top and
    bottom of the structure are subject to Dirichlet boundary
    conditions for the eletrostatic potential $\varphi$, the
    remaining boundary carries Neumann boundary conditions
    (lightly shaded; the frontal area is kept transparent). A
    triple quantum well structure induced by different material
    layers is indicated in the lower part, corresponding to the
    doping $\fd$.}
\end{figure}

\begin{assu}[Geometry, extended]
  \label{ass:geometry-extend}
  The set
  $\Omega \subset \R^3$ is a bounded domain and satisfies the
  \emph{thickness} condition: There exist constants $0 < c \leq C < 1$
  such that
  \begin{equation}
    \label{eq:thickness}
    c \leq \frac{|B_r(x) \cap \Omega|}{|B_r(x)|} \leq C \qquad (x \in
    \partial\Omega, \, r \in (0,1]).
  \end{equation}
  Moreover, the
  following additional properties hold true for the boundary $\partial\Omega$:
  \begin{enumerate}
  \item $D \subseteq \partial\Omega$ is a closed $(d-1)$-set with
    $\cH_2(D) > 0$. The relative boundary $\partial D$ of $\Diri$ in
    $\partial\Omega$ is a $(d-2)$-set.
  \item There are Lipschitz coordinate charts available around
    $\overline{\partial\Omega \setminus D}$, that is, for every
    $x \in \overline{\partial\Omega \setminus D}$, there is an open
    neighborhood $\cU$ of $x$ and a bi-Lipschitz mapping
    $\phi_x \colon \cU \to (-1,1)^d$ such that $\phi_x(x) = 0$ and
    $\phi_x(\cU\cap\Omega) = (-1,0) \times (-1,1)^{d-1}$. 
  \item $ \Pi \subset \Omega$ is a Lipschitz surface, not necessarily
    connected, which forms a $(d-1)$-set.
  \end{enumerate}
\end{assu}

\begin{rem}
  \label{rem:extended-stronger}
  We emphasize the condition $C<1$ in the thickness
  condition~\eqref{eq:thickness} in the foregoing
  assumption. This requirement makes the thickness condition
  strictly stronger than the \emph{interior} thickness condition
  for $\partial\Omega$ which is equivalent $\Omega$ being
  $d$-regular as mentioned in
  \cref{rem:ahlfors-thickness}. In fact, the thickness
  condition~\eqref{eq:thickness} implies that both $\Omega$ and
  $\Omega^c$ are $d$-regular~(\cite[Ex.~2.4]{BM19}). In
  particular, \cref{ass:geometry-extend} always
  implies \cref{a-geometry}.
\end{rem}

\cref{ass:geometry-extend} defines the general geometric
framework for this section which however is restricted implicitly
by \cref{assu:iso} below. We are convinced that
this setting is sufficiently broad to cover (almost) all relevant
semiconductor geometries, in particular in view of the arrangement of
$\Diri$ and $\Neumann$. Please see also the more elaborate
\cref{r-kommentar} on this topic below. 

The second-order (elliptic) differential operators occurring
in~\eqref{vanRoos} will of course be considered in their weak form
introduced in \cref{def:div-grad} with the Robin boundary
form realized as in \cref{def:boundary-operator}. We pose
the following assumptions on their data: 

\begin{assu} \label{assu:iso}  We have $\eps,\mu_1,\mu_2 \in
  \cC(c_\bullet,c^\bullet)$ and $\eps_\Neumann \in
  L^\infty(\Neumann;\cH_2)$ and all these functions are
  \emph{real}. Moreover, the following additional properties hold true:
  \begin{enumerate}[(i)]
  \item \label{item:assu:iso:iso} There is a common integrability exponent $q \in (3,4)$ such
    that 
    \begin{align} \label{e-iso1} -\nabla \cdot \varepsilon \nabla +
      \tr_\Neumann^* \eps_\Neumann \tr_\Neumann
      & \in \cLiso\bigl(W^{1,q}_D(\Omega)\to W^{-1,q}_D(\Omega)\bigr)
    \intertext{and}
    \label{e-iso02} -\nabla \cdot \mu_k \nabla
    & \in \cLiso\bigl(W^{1,q}_D(\Omega) \to W^{-1,q}_D(\Omega)\bigr) \quad (k=1,2).
    \end{align} 
  \item\label{item:assu:iso:mult} There is
    $\vartheta \in (0,1-\frac3q)$ such that
    $\varepsilon_{ij} \in \cM(H^{\vartheta,q}(\Omega))$ and
    $(\mu_1)_{ij}, (\mu_2)_{ij} \in \cM(H^{\vartheta,q}(\Omega))$.
  \end{enumerate}
\end{assu}

See also \cref{def:coefficient-functions,def:multiplier} for
the $\cC(c_\bullet,c^\bullet)$ and multiplier notions. Note
moreover that due to the assumption $q \in (3,4)$, we have
$1-\frac3q < \frac1q = \frac1q \wedge \frac1{q'}$. Finally, we
point out that while we pose quite similar assumptions on $\eps$
and $\mu_1,\mu_2$, the assumptions are used in a quite different
way. For
$-\nabla\cdot\eps\nabla+ \tr_\Neumann^* \eps_\Neumann
\tr_\Neumann$, they enable us to use the extrapolated elliptic
regularity result in \cref{thm:invert-extrapol}. For
$-\nabla \cdot \mu_k \nabla$, the isomorphism
assumption~\eqref{e-iso02} will allow to determine the domains
of certain fractional powers of these operators which are of
interest for classical parabolic theory for semilinear equations
such as~\eqref{CuCo-eq}, see \cref{l-ineter} below. On the other
hand, the multiplier assumption on $\mu_k$ is used to deal with
the drift-structure induced by the fluxes $j_k$ as defined
in~\eqref{eq:curr-dens}.

Whenever we refer to the integrability $q$ from now on, a fixed
number from \cref{assu:iso} is meant.  

\begin{rem} \label{r-kommentar}
  \begin{enumerate}[(i)]
  \item Properties~\eqref{e-iso1} and~\eqref{e-iso02} remain true for
    all $\tilde q \in [2,q)$ by the Lax-Milgram lemma and
    interpolation
    (\cref{prop:interpolation-zero-trace}). In particular,
    the set of indices $q \geq 2$ such that~\eqref{e-iso1}
    and~\eqref{e-iso02} holds true always forms an interval. Thus it
    is sufficient to know that each of the operators~\eqref{e-iso1}
    and~\eqref{e-iso02} is an isomorphism for \emph{some} $q>3$ in
    order to find a \emph{common} $q$. Let us moreover note that in
    the presence of mixed boundary conditions one cannot expect
    $q \geq 4$ in \cref{assu:iso}~\ref{item:assu:iso:iso}
    when $\Diri$ and $\Neumann$ meet due to the counterexample by
    Shamir~\cite[Introduction]{shamir}.
  \item \cref{assu:iso}~\ref{item:assu:iso:iso} is fulfilled
    by very general classes of layered structures and additionally, if
    $\Diri$ and its complement $\Neumann$ do not meet in a too wild
    manner, for the most relevant model
    settings. (See~\cite{RoChristJo} for the latter.) A global
    framework has recently been established in~\cite{disser}.
    However, \cref{assu:iso}~\ref{item:assu:iso:iso} is
    indeed a restriction on the class of admissible coefficient
    functions $\varepsilon$ and $\mu_k$. For instance, it is typically
    not satisfied if three or more different materials meet at one
    edge.
  \item Note that it is typically not restrictive to assume that all
    three differential operators in~\eqref{e-iso1} and~\eqref{e-iso02}
    provide topological isomorphisms \emph{at once} if one of them does, since this
    property mainly depends on the (possibly) discontinuous
    coefficient functions versus the geometry of $\Diri$. This is
    determined by the material properties of the device $\Omega$,
    i.e., the coefficient functions $\mu_1, \mu_2, \varepsilon$ will
    often exhibit similar discontinuities and degeneracy.
  \item The multiplier assumption in
    \cref{assu:iso}~\ref{item:assu:iso:mult} is a very broad
    one and certainly fulfilled in the context of realistic
    semiconductor structures. Recall that, as seen in
    \cref{rem:multipliers}, the multiplier assumptions on
    $\mu_1,\mu_2$ and $\eps$ hold in fact for \emph{all}
    differentiability orders $\tau \in [0,\vartheta]$.
  \end{enumerate}
\end{rem}

\subsubsection[Gradient-dependent and interfacial
recombination]{Assumptions on recombination terms}
\label{Reco}
%
We next give the assumptions for the recombination terms
$\reaction^\Omega,$ $\reaction^\Pi, r^\Neumann$
in~\eqref{CuCo-eq}. For convenience, we introduce
\begin{equation*}
  \bfW^{1,q}_D(\Omega) \defn W^{1,q}_D(\Omega) \times
      W^{1,q}_D(\Omega).
\end{equation*}
Note that by
\emph{locally Lipschitzian} we mean that the corresponding
function is Lipschitz continuous on bounded sets.

\begin{assu} \label{assu:recomb} All reaction terms
  $\reaction^\Omega,$ $\reaction^\Pi, r^\Neumann$ map
  \emph{real} functions to again \emph{real ones}. Moreover:
  \begin{enumerate}[(i)]
  \item The bulk reaction term $\reaction^\Omega$ is a locally
    Lipschitzian mapping
    \begin{equation*}
      \reaction^\Omega \colon \bfW^{1,q}_D(\Omega) \times
      W^{1,q}_D(\Omega)   \ni (u, 
      \varphi) \mapsto r^\Omega(u, \varphi) \in L^{\frac {q}{2}}(\Omega). 
    \end{equation*} 
  \item The reaction term $\reaction^\Neumann$  on $\Neumann$ is
  a locally Lipschitzian mapping
  \begin{equation*}
    r^\Neumann \colon\bfW^{1,q}_D(\Omega) \times
    W^{1,q}_D(\Omega) \ni (u, \varphi) \mapsto r^\Neumann(u,
    \varphi) 
    \in L^{4}(\Neumann;\sigma).
  \end{equation*}
\item The interfacial reaction term $\reaction^\Pi$ on $\Pi$ satisfies
  the same assumption as $\reaction^\Neumann$ does, \emph{mutatis mutandis}.
  \end{enumerate}  
\end{assu}

The choice of integrability $4$ on $\Gamma$ and $\Pi$, respectively,
is connected to $q < 4$ in \cref{assu:iso}. This can be seen
in \cref{t-propX} below.

It is easy to see that the recombination terms $r^\Omega_{\text{SRH}}$
and $r^\Omega_{\text{Auger}}$ introduced in~\eqref{e-shockleyrh}
and~\eqref{e-auger} satisfy \cref{assu:recomb}. On the other
hand, validating the same for the Avalanche generation term, depending on the
electric field $\varphi$ and the currents $j_k$, is nontrivial, but we
indeed find:
\begin{lemma} \label{l-Lipsch} The Avalanche recombination term
  $r^\Omega_{\text{Ava}}$ defined in~\eqref{e-aval} satisfies
  \cref{assu:recomb}.
\end{lemma}
\begin{proof}
  The lemma is proved in~\cite[Ch.~3.4]{vanroos3d}. More precisely,
  the current densitites
  \begin{equation*}
    \bfW^{1,q}_D(\Omega) \times W^{1,q}_D(\Omega)   \ni (u,
    \varphi) \mapsto j_k = \mu_k\bigl(\nabla u_k + (-1)^ku_k
    \nabla\varphi\bigr) \in L^q(\Omega)
  \end{equation*}
  are locally Lipschitz continuous via the estimate
  \begin{multline*}
    \bigl\|j_k(u,\varphi)-j_k(v,\psi)\bigr\|_{L^q(\Omega)} \leq
    \|\mu_k\|_{L^\infty(\Omega)}\left[\|\nabla\psi\|_{L^q(\Omega)}
      \bigl\|u_k-v_k\bigr\|_{L^{\infty}(\Omega)}  
    \right. \\ + \left. \bigl\|\nabla u_k - \nabla
      v_k\bigr\|_{L^q(\Omega)} +
      \|u_k\|_{L^\infty(\Omega)}\bigl\|\nabla\varphi
      -\nabla\psi\bigr\|_{L^q(\Omega)}\right]
  \end{multline*}
  and the embedding $W^{1,q}_D(\Omega) \embeds L^\infty(\Omega)$ due to
  $q>d=3$. It remains to connect this with~\cite[Lem.~3.9]{vanroos3d}
  where
  \begin{multline*}
    \bigl\|r^\Omega_{\text{Ava}}(u,\varphi) -
    r^\Omega_{\text{Ava}}(v,\psi)\bigr\|_{L^{\frac{q}2}(\Omega)}
    \\\lesssim 
    \|\nabla\varphi\|_{L^q(\Omega)}\left(\bigl\|j_1(u,\varphi)
    - j_1(v,\psi)\bigr\|_{L^q(\Omega)}+\bigl\|j_2(u,\varphi)
    - j_2(v,\psi)\bigr\|_{L^q(\Omega)}\right) \\ + 
    \left(\|j_1(v,\psi)\|_{L^q(\Omega)}
    +\|j_2(v,\psi)\|_{L^q(\Omega)}\right)\bigl\|\nabla\varphi
    - \nabla\psi\bigr\|_{L^q(\Omega)}
  \end{multline*}
  is shown.
\end{proof}

\begin{rem}
  \label{rem:similarity-quad}
  It is imperative to compare the very last estimate in the foregoing proof
  to the Lipschitz estimate for the quadratic gradient function
  \begin{equation*}
    \bigl\||\nabla v_1|^2 - |\nabla
    v_2|^2\bigr\|_{L^{\frac{q}2}(\Omega)} \leq \left(\|\nabla
    v_1\|_{L^q(\Omega)}+\|\nabla v_2\|_{L^q(\Omega)}\right) \|\nabla
    v_1-\nabla v_2\|_{L^q(\Omega)},
  \end{equation*}
  which is of very similar structure. This is the connection to the
  quadratic gradient nonlinearity $v \mapsto |\nabla v|^2$ which was
  mentioned in the introduction.
\end{rem}

\subsubsection{Assumptions on auxiliary data}

Lastly, we give the assumptions on the doping $\fd$. It permits dopings
which live in the bulk \emph{and}, possibly, on $2$-dimensional
surfaces, see \cref{t-propX} below. We comment on the actual
requirement in \cref{rem:doping-assumption} below.
\begin{assu} \label{assu:poi-diri-bv} The doping $\fd$ belongs to the
  space $H^{-\frac {3}{q },q}_D(\Omega)$.
\end{assu}

\subsection{Existence and uniqueness for the abstract semilinear equation}\label{sec:exist-uniq-abstr}

It was already explained in the introduction that we intend to
solve the van Roosbroeck system~\eqref{vanRoos} by eliminating
the electrostatic potential $\varphi$ in~\eqref{CuCo-eq} and
\eqref{eq:curr-dens} as a function of the densities $u$, thereby
considering~\eqref{CuCo-eq} as a \emph{semilinear} parabolic
equation in the densities. Having this in mind, we give a brief
discussion on the question which Banach space
$\bfX = X \oplus X$ will be adequate to consider this parabolic
equation in, based on the structural- and regularity properties
of the unknowns $u,\varphi$ and the data such as $\fd$.

\begin{itemize}
\item In view of the jump condition on the surface $\Pi$ on the fluxes
  $j_k$ in~\eqref{CuCo-eq}, it cannot be expected that
  $\dive j_k$ is a function.  This excludes spaces of type
  $L^p(\Omega)$.  In addition, the space $X$ should be large
  enough to include
  distributional objects, so that the the inhomogeneous Neumann datum
  $r^\Neumann$ in the current-continuity equations~\eqref{CuCo-eq} and
  the surface recombination term $r^\Pi$ can be included in the
  right-hand side of the current continuity equations.
\item For our analysis, we require an adequate parabolic theory
  for the divergence operators on $\bfX$.  Due to the
  \emph{non-smooth geometry}, the \emph{mixed boundary
    conditions} and \emph{discontinuous coefficient functions},
  this is nontrivial. The minimum needed is that the operators
  $\nabla \cdot \mu_k \nabla$ generate \emph{analytic} semigroups
  on $X$.
\item For the handling of the squared gradient nonlinearity or other
  functions of gradients in the Avalanche and other recombination
  terms, it is imperative to have $\nabla u_k(t)$ in $L^q(\Omega)$
  \emph{in every time point} $t$ at ones disposal in order to apply
  standard semilinear parabolic theory, see e.g.~\cite[Ch.~3.3]{henry}
  or~\cite[Ch.~7]{luna}. Hence, the Banach space $X$ needs to be such
  that an interpolation space between the domain of
  $\nabla\cdot\mu_k\nabla$ in $X$ and $X$ itself embeds continuously
  into $W^{1,q}(\Omega)$. But this excludes spaces of type
  $X= W_D^{-1,q}(\Omega)$ since the domain of $\nabla\cdot\mu_k\nabla$
  there is at best $W^{1,q}_D(\Omega)$ 
  (\cref{assu:iso}~\ref{item:assu:iso:iso}). With this
  strategy, at the same time, the space $X$ needs to be sufficiently
  large for the embedding $L^{q/2}(\Omega) \embeds X$ to hold to
  include the pointwise quadratic gradient.
\end{itemize}

We will choose $X$ as an interpolation space between
$W^{-1,q}_D(\Omega)$ and $L^q(\Omega)$. This will yield a
framework in which the requirements listed above are indeed
satisfied, see \cref{t-propX,lem:div-op-X,l-ineter} below.

To this end, we first quote the \emph{nonsymmetric}
interpolation result which will allow us to identify the
designated (interpolation) space $X$ with a space from the
Bessel scale. This proposition is the only point where the
strengthened geometric assumptions in
\cref{ass:geometry-extend} compared to
\cref{a-geometry} are needed. The primal interpolation
result is quoted from~\cite{BM19}, and the dual scale is
obtained in the same manner as done for proof of
\cref{cor:interpolation-dual}.

\begin{proposition}[{Interpolation~\cite[Thm.~1.1]{BM19}}]
  \label{prop:nonsymmetric-interpolation}
  Let $p \in (1,\infty)$ and $\theta \in (0,1)$, and let
  $E \subset \overline\Omega$ be a $(d-1)$-set. Then
  \begin{equation*}
    \bigl[W^{1,p}_E(\Omega),L^p(\Omega)\bigr]_{\theta} =
    \begin{cases}
      H^{1-\theta,p}_E(\Omega) & \text{if}~\theta < 1-\frac1p
      \\[0.5em]
      H^{1-\theta,p}(\Omega) & \text{if}~\theta > 1-\frac1p
    \end{cases}
  \end{equation*}
  and accordingly
   \begin{equation*}
     \bigl[W^{-1,p}_E(\Omega),L^p(\Omega)\bigr]_{\theta} =
     \left(\bigl[W^{1,p'}_E(\Omega),L^{p'}(\Omega)\bigr]_{\theta}\right)^\star
     = 
    \begin{cases}
      H^{\theta-1,p}_E(\Omega) & \text{if}~\theta < \frac1p
      \\[0.5em]
      H^{\theta-1,p}(\Omega) & \text{if}~\theta > \frac1p.
    \end{cases}
  \end{equation*}
\end{proposition}

Moreover, let us reiterate the following immediate consequence of
\cref{assu:iso} and \cref{thm:invert-extrapol}, where
$\vartheta$ is the number from
\cref{assu:iso}~\ref{item:assu:iso:mult}:

\begin{lemma}\label{lem:sneiberg-follow}
  There is a number $\bar s \in (0,\vartheta]$ such that the operator
  $-\nabla\cdot\eps\nabla + \tr_\Neumann^*\eps_\Neumann\tr_\Neumann$
  is a topological isomorphism between $H^{1+s,q}_D(\Omega)$ and
  $H^{s-1,q}_D(\Omega)$ for all $s \in [0,\bar s)$.
\end{lemma}

Finally, we define the Banach space $\bfX$ in which we intend to 
investigate the parabolic equation: 

\begin{definition} \label{d-X} Let $\bar s$ be the number from
  \cref{lem:sneiberg-follow}. We fix $\tau \in (0,\bar s)$ and
  define
  \begin{equation*}
    X \defn \bigl[L^q(\Omega),W^{-1,q}_D(\Omega)\bigr]_{1-\tau,q} = H^{\tau-1,q}_D(\Omega)  \quad
    \text{and} \quad \bfX :=X \oplus X.
  \end{equation*}
  The identity of the interpolation space and $H^{\tau-1,q}_D(\Omega)$
  follows from \cref{prop:nonsymmetric-interpolation}.
\end{definition}

\begin{rem}\label{rem:doping-assumption}
  Due to the assumptions on $\vartheta$, we have 
  $\tau \in (0,1-\frac3q)$. In particular, $\tau-1 < -3/q$, thus
  $H^{-3/q,q}(\Omega) \embeds H^{\tau-1,q}_D(\Omega) = X$, and so
  $\fd \in X$ by \cref{assu:poi-diri-bv}.
\end{rem}

It remains to verify that $X$ or $\bfX$ satisfy the requirements we
established above. The  first lemma joins 
\cref{rem:doping-assumption} in showing that $X$ is sufficiently
large for our means.

\begin{lemma} \label{t-propX} There holds $L^{\frac{q}2}(\Omega)
  \embeds X$. Moreover, the adjoint trace mappings $\tr_\Neumann^* \colon
  L^4(\Neumann;\cH_2) \to X$ and $\tr_\Pi^* \colon L^4(\Pi;\cH_2) \to
  X$ give rise to continuous embeddings. 
\end{lemma}

\begin{proof} The first embedding follows from taking the adjoint of
  the Sobolev embedding
  $H^{3/q,q'}_D(\Omega) \embeds L^{\frac{q}{q-2}}(\Omega)$ and the
  observation in \cref{rem:doping-assumption}.

  Continuity of the adjoint trace is proven
  in~\cite[Lem.~4.4]{vanroos3d} by showing that
  \begin{equation}
    \tr_\Neumann \colon H^{\frac{3}{q},q'}_D(\Omega) \to
    L^{\frac43}(\Neumann;\cH_2)  
    \quad \text{and} \quad \tr_\Pi
    \colon H^{\frac {3}{q},q'}_D(\Omega) \to
    L^{\frac43}(\Pi;\cH_2) \label{eq:trace-bessel}    
  \end{equation}
  are continuous, and then taking adjoints. We give a quick
  additional proof of~\eqref{eq:trace-bessel} based on the trace
  theorem from \cref{cor:slobo-domain-trace-op}: The
  condition $q \in (3,4)$ implies that $\frac1{q'} < \frac3q$,
  hence we can find $s \in (\frac1{q'},\frac3q)$ so that
  $H^{3/q,q'}(\Omega) \embeds W^{s,q'}(\Omega)$. Now
  \cref{cor:slobo-domain-trace-op} gives the result
  because it says that $\tr_\Neumann$ maps $W^{s,q'}(\Omega)$
  continuously into $L^{q'}(\Neumann)$ when $s > \frac1{q'}$; it
  remains only to observe that $q' > \frac43$. The reasoning for
  $\tr_\Pi$ is completely analogous because
  \cref{cor:slobo-domain-trace-op} is valid for
  $(d-1)$-regular sets $E \subset \overline\Omega$.
\end{proof}

\cref{t-propX} puts us in the position to establish the
functional-analytic setting for the van Roosbroeck
system~\eqref{vanRoos}. Recall also \cref{lem:sneiberg-follow}.

\begin{definition}[Solution concept]
  \label{def:func-ana-defn}
  Define the mapping $v \mapsto \varphi(v)$ by
  \begin{equation}\label{eq:ellipt-regu-potential}
    v \mapsto \varphi
    \defn \left(-\nabla\cdot\eps\nabla +
      \tr_\Neumann^*\eps_\Neumann\tr_\Neumann\right)^{-1}\bigl(\fd -
    v_1 + v_2\bigr) 
  \end{equation}
  and set
  \begin{equation*}
    r(v) \defn r^\Omega\bigl(v,\varphi(v)\bigr) + \tr_\Neumann^*
    r^\Neumann\bigl(v,\varphi(v)\bigr) + \tr_\Pi^* r^\Pi\bigl(v,\varphi(v)\bigr).
  \end{equation*}
  Then we say that a function $u = (u_1,u_2) \colon
  [0,T^\bullet) \to \bfX$ 
  is a \emph{solution to the van Roosbroeck
    system}~\eqref{vanRoos}, if $u(0) = u^0$ and
  \begin{equation*}
    u_k'(t) - \nabla \cdot \mu_k \nabla u_k(t) = (-1)^{k+1} \nabla\cdot
    u_k(t)\mu_k\nabla \varphi(u(t)) + r(u(t)) \quad \text{in}~X \quad (k=1,2)
  \end{equation*}
  for all $t \in (0,T_\bullet)$, where $T_\bullet \in (0,T]$.
\end{definition}

Before we prove existence and uniqueness of a solution in the
sense of \cref{def:func-ana-defn}, we further collect
some results about the elliptic operators
$-\nabla\cdot\mu_k\nabla$. In the second part, we make use of
the co-restriction of
$-\nabla\cdot\mu_k\nabla \colon W^{1,q}_D(\Omega) \to
W^{-1,q}_D(\Omega)$ to $L^q(\Omega)$, considered as a closed
operator in that space, and analogously for $X$.

\begin{lemma} \label{lem:div-op-X}
  \begin{enumerate}[(i)]
  \item The square root
    $(-\nabla\cdot\mu_k\nabla)^{-1/2}$ provides a topological
    isomorphism between $ W^{-1,q}_D(\Omega)$ and $L^q(\Omega)$.
    \label{item:div-op-X-a}
  \item The operators $\nabla \cdot \mu_k \nabla$ are generators
    of analytic semigroups and their negatives admit bounded
    imaginary powers on $L^q(\Omega)$ space, on
    $W^{-1,q}_D(\Omega)$, and also on
    $X$. \label{item:div-op-X-b}
  \end{enumerate}
\end{lemma}
\begin{proof}
  \ref{item:div-op-X-a} is~\cite[Thms.~1.2/1.6]{E19sys}, see
  also~\cite[Thm.~5.1]{auscher}. \ref{item:div-op-X-b}: The
  proof for both properties works in the same way: First, the
  property is established on $L^q(\Omega)$, then the square root
  isomorphism from~\ref{item:div-op-X-a} is used to transfer the
  property to $W^{-1,q}_D(\Omega)$, and the $X$ case is finally
  obtained by interpolation.

  For the
  generator property on $L^q(\Omega)$, we refer
  to~\cite[Thm.~3.1]{tERL19}
  and carry over the equivalent resolvent
  estimates~(\cite[Thm.~1.45]{Ouh05}) to
  $W^{-1,q}_D(\Omega)$. Interpolation is then easy.

  Regarding bounded imaginary powers, we refer
  to~\cite[Cor.~3.4]{tERL19} for the $L^q(\Omega)$ case. The transfer
  to $W^{-1,q}_D(\Omega)$ is provided
  by~\cite[Prop.~2.11]{DHP03}. Finally, interpolation works due
  to~\cite[Cor.~7.1.17]{haase}.
\end{proof}

We finally determine the domain of a particular fractional power of
$-\nabla\cdot\mu_k\nabla$ to be $W^{1,q}_D(\Omega)$ which is one of
the cornerstones in the treatment of equations with nonlinear gradient
terms. Here, $\dom_X(-\nabla\cdot\mu_k\nabla)$ denotes the domain of
the corestriction of $-\nabla\cdot\mu_k\nabla$ to $X =
H^{\tau-1,q}_D(\Omega) \subset W^{-1,q}_D(\Omega)$.

\begin{lemma} \label{l-ineter}
  One has
  \begin{equation} \label{eq:frac-power-w1q} \bigl[\dom_ X(-\nabla \cdot \mu_k
    \nabla ),X\bigr]_{\frac\tau2} = \dom_X\bigl ((-\nabla \cdot \mu_k
    \nabla )^{1 - \frac\tau2} \bigr ) = W^{1,q}_D(\Omega).
  \end{equation}
\end{lemma}
\begin{proof}
  The first equality in~\eqref{eq:frac-power-w1q} follows
  from~\cite[Ch.~1.15.3]{triebel} due to the \emph{bounded imaginary
    powers} property of $-\nabla\cdot\mu_k\nabla$ provided by
  \cref{lem:div-op-X}. Moreover, without loss of generality reversing the interpolation
  order, we have
  \begin{align*}
    X = \bigl[W^{-1,q}_D(\Omega),L^q(\Omega)\bigr]_{\tau} & =
    \bigl[W^{-1,q}_D(\Omega),\dom_{W^{-1,q}_D(\Omega)}\bigl ((-\nabla \cdot \mu_k \nabla
    )^{1/2} \bigr )\bigr]_{\tau} \\ &= \dom_{W^{-1,q}_D(\Omega)}\bigl ((-\nabla \cdot
    \mu_k \nabla )^{\tau/2} \bigr ).
  \end{align*}
  Now use \cref{assu:iso} and apply
  $(-\nabla \cdot \mu_k \nabla)^{-1} \in \cLiso(W^{-1,q}_D(\Omega) \to
  W^{1,q}_D(\Omega))$ to obtain the second equality in~\eqref{eq:frac-power-w1q}.
\end{proof}

We are not able to formulate and prove the main result.

\begin{theorem}[Local-in-time wellposedness] \label{t-formulat}
  Suppose that $u^0 = (u_1^0,u_2^0) \in
  \bfW^{1,q}_D(\Omega)$. Then the van Roosbroeck
  system~\eqref{vanRoos} admits a unique classical local-in-time
  solution $u$ in the sense of
  \cref{def:func-ana-defn}. That is, there is
  $T_\bullet \in (0,T]$ such that
  \begin{equation*}
    u \in C^{1-\frac\tau2}\bigl([0,T_\bullet];\bfX\bigr) \cap
    C\bigl([0,T_\bullet];\bfW^{1,q}_D(\Omega)\bigr) \cap 
    C^1\bigl((0,T_\bullet];\bfX\bigr). 
  \end{equation*}
  The mapping $u^0 \mapsto u$ is Lipschitz continuous from a
  neighbourhood of $u^0$ in $\bfW^{1,q}_D(\Omega)$ to
  $C([0,T_\bullet];\bfX)$.  Moreover, if $u^0$ is real, then $u$ is
  \emph{real} on the interval of existence.
\end{theorem}
\begin{proof}
  With the preparationary work done, we can rely on standard
  semilinear parabolic theory as established
  in~\cite[Ch.~3.3]{henry},~\cite[Ch.~6.3]{pazy}
  or~\cite[Ch.~7]{luna} to obtain the local-in-time solution
  with the announced regularity. Indeed, we already know that
  each of the operators $\nabla \cdot \mu_k \nabla$ generates a
  semigroup which is analytic on $X$. Clearly, the diagonal
  operator matrix $\cA$ induced by $\nabla \cdot \mu_k \nabla$
  then also generates an analytic semigroup on $\bfX$. It
  remains to establish that the right-hand sides in the reduced
  problem as defined in \cref{def:func-ana-defn} are
  locally Lipschitz continuous on the $\bfX$-domain of a true
  fractional power $\cA^\alpha$ of $\cA$. In view
  \cref{l-ineter}, we focus on $\alpha = 1-\frac\tau2$ and on
  obtaining the Lipschitz property on $\bfW^{1,q}_D(\Omega)$. This is
  also compatible with the assumed initial value regularity. (Here,
  note that $\dom_\bfX\cA$ is dense in $\bfW^{1,q}_D(\Omega)$ due to
  the interpolation identity~\eqref{eq:frac-power-w1q}.)

  For the reaction terms $r^\Omega, r^\Neumann, r^\Pi$, this is by
  \cref{assu:recomb} and \cref{l-Lipsch}. We only need
  to consider the drift-diffusion terms. It is clear that
  \begin{equation}\label{eq:ellipt-regu-lipschitz}
    \bfW^{1,q}_D(\Omega) \ni v \mapsto \varphi(v) =
    \left(-\nabla\cdot\eps\nabla + 
      \tr_\Neumann^*\eps_\Neumann\tr_\Neumann\right)^{-1}\bigl(\fd -
    v_1 + v_2\bigr) 
    \in H^{1+\tau,q}_D(\Omega)
  \end{equation}
  as defined in~\eqref{eq:ellipt-regu-potential} is Lipschitz
  continuous, recall \cref{lem:sneiberg-follow} and
  \cref{rem:doping-assumption}. Thus, quite similar to the
  estimate in the proof of \cref{l-Lipsch}, we obtain for
  $v,w \in \bfW^{1,q}_D(\Omega)$:
  \begin{multline}
    \bigl\|\nabla \cdot u_k \mu_k \nabla \varphi(w)- \nabla \cdot v_k
    \mu_k \nabla \varphi(v)\bigr\|_X \\ =   \bigl\|\nabla \cdot w_k
    \mu_k \nabla \bigl(\varphi(w)-\varphi(v)\bigr) - \nabla
    \cdot \bigl(v_k-w_k\bigr) 
    \mu_k \nabla \varphi(v)\bigr\|_X\label{eq:lipschitz-to-estimate}
  \end{multline}
  and of course we split the latter with the triangle inequality. From
  there, we rely on~\eqref{eq:ellipt-regu-lipschitz} and
  multiplier properties of $\mu_k$ and 
  $w_k$. This is because if $\omega \in \cM(H^{s,q}(\Omega))$ and $\psi
  \in H^{1+s,q}_D(\Omega)$ for some $s \in (0,\frac1q)$, then using
  \cref{lem:gradient-extension} and estimating as in the proof of
  \cref{lem:div-grad-extend}, we find
  \begin{equation}
    \label{eq:crucial-estimate}
    \bigl\|\nabla\cdot \omega\nabla \psi\bigr\|_{H^{1-s,q}_D(\Omega)} \leq
    \|\omega\|_{\cM(H^{s,q}(\Omega))} \|\psi\|_{H^{1+s,q}_D(\Omega)},
  \end{equation}
  and $H^{1+s,q}_D(\Omega)$ is the biggest space for $\psi$ we
  can determine for which such an estimate works. We had in fact
  assumed that $\mu_k$ is a multiplier on $H^{\tau,q}_D(\Omega)$
  in \cref{assu:iso}~\ref{item:assu:iso:mult}. For
  $w_k$, we observe that
  $W^{1,q}_D(\Omega) \embeds C^{1-3/q}(\Omega)$ and
  $\tau < 1-3/q$ by assumption, see
  \cref{rem:doping-assumption}. Hence
  $C^{1-3/q}(\Omega) \embeds \cM(H^{\tau,q}(\Omega))$ as noted
  in \cref{rem:multipliers} and
  $u_k \in W^{1,q}_D(\Omega)$ is also a multiplier on
  $H^{\tau,q}(\Omega)$. Thus, via~\eqref{eq:crucial-estimate}
  \begin{align}
    \notag
    \bigl\|\nabla \cdot w_k
    \mu_k \nabla \bigl(\varphi(w)-\varphi(v)\bigr)\bigr\|_X& \leq
    \|w_k\|_{\cM(H^{\tau,q}(\Omega))}\|\mu_k\|_{\cM(H^{\tau,q}(\Omega))}
    \bigl\|\varphi(w)-\varphi(v)\bigr\|_{H^{1+\tau,q}_D(\Omega)} \\ &
    \lesssim
    \|w_k\|_{W^{1,q}_D(\Omega)}
    \bigl\|\varphi(w)-\varphi(v)\bigr\|_{H^{1+\tau,q}_D(\Omega)}
        \label{eq:lipschitz-first-estimate}
  \end{align}
  In a similar fashion, the second term is estimated by
  \begin{align}
    \notag
    \bigl\|\nabla \cdot \bigl(v_k-w_k\bigr)
    \mu_k \nabla \varphi(v)\bigr\|_X  &\leq
    \|\mu_k\|_{\cM(H^{\tau,q}(\Omega))} 
    \|\varphi(v)\|_{H^{1+\tau,q}_D(\Omega)}\bigl\|w_k-v_k\bigr\|_{\cM(H^{\tau,q}(\Omega))}
    \\ & \lesssim 
    \|\varphi(v)\|_{H^{1+\tau,q}_D(\Omega)}
    \bigl\|w_k-v_k\bigr\|_{W^{1,q}_D(\Omega)} \label{eq:lipschitz-second-estimate}   
  \end{align}
  Estimating~\eqref{eq:lipschitz-to-estimate} further
  using~\eqref{eq:lipschitz-first-estimate}
  and~\eqref{eq:lipschitz-second-estimate} and using Lipschitz
  continuity of $v \mapsto \varphi(v)$, we obtain the desired
  \emph{local} Lipschitz continuity on
  $\bfW^{1,q}_D(\Omega)$. Hence standard semilinear theory as in
  the works mentioned at the beginning of the proof shows that a
  solution $u$ to the semiconductor equations in the sense of
  \cref{def:func-ana-defn} with the announced exists
  locally in time.

  Finally, let us show that this solution $u$ is indeed a \emph{real}
  one. In fact, this is implied by the following facts:
  \begin{enumerate}[(i)]
  \item The semigroups generated by $\nabla \cdot \mu_k \nabla$ are
    \emph{real} ones, that is, they transform elements from the real
    part of $W^{-1,q}_D(\Omega)$ into \emph{real}
    functions. (See~\cite[Ch.~2.2/4.2]{Ouh05}.) Clearly, this
    transfers to $\cA$ on $\bfX$.
  \item Since the initial values $u_1^0$ and $u_2^0$ were supposed to
    be \emph{real}, the fixed
    point procedure used to construct a solution in the classical
    proof in~\cite[Thm.~6.3.1]{pazy} can in fact be done in the
    \emph{real} part of $\bfX$.
  \end{enumerate}
  This completes the proof.
\end{proof}

\begin{rem} \label{r-concl} 
  \begin{enumerate}[(i)]
  \item Let us point out that the Lipschitz estimate in the proof of
    the main \cref{t-formulat} \emph{only} works so smoothly
    using~\eqref{eq:crucial-estimate} because we in fact know
    that~\eqref{eq:ellipt-regu-lipschitz} holds with the
    $H^{1+\tau,q}_D(\Omega)$ image space, which in turn is a
    consequence of extrapolated elliptic regularity as established in
    \cref{thm:invert-extrapol}, see
    \cref{lem:sneiberg-follow}. It was already mentioned in
    the foregoing proof that $H^{1+\tau,q}_D(\Omega)$ is exactly
    the largest space for which an
    estimate of the form~\eqref{eq:crucial-estimate} can work
    with $\omega = w_k \mu_k$. Note here that $w_k$
    is not fixed and does not necessarily admit a strictly
    positive lower bound.
  \item The presented real world example is one among many
    others which can be treated the same way.  We focused
    here---in contrast to~\cite{vanroos3d}---on the case where
    the chemical potential and the densities in the
    semiconductor model are related by Boltzmann statistics,
    i.e., where their relating function is the exponential (or
    logarithm, depending on the point of view). This has the
    consequence that the resulting evolution equation for the
    densities is a \emph{semilinear} one. In the general case of
    Fermi-Dirac statistics, the corresponding evolution equation
    will be a \emph{quasilinear} one. However, such a
    quasilinear equation can also be treated in a quite similar
    manner to the above. One would use Pr\"uss' pioneering
    theorem~(\cite{pruess}) as the abstract tool, based on the
    fact that the operators $-\nabla \cdot \mu_k \nabla$ in fact
    even satisfy maximal parabolic regularity on the spaces
    $X = H^{\tau-1,q}_D(\Omega)$, see~\cite[Ch.~11]{auscher}
    and~\cite[Lemma~5.3]{RoJo}. The analysis above shows that
    exactly the extrapolation result
    \cref{thm:invert-extrapol} allows to eliminate the
    electrostatic potential implicitly, in a very much simpler
    way as done before, compare~\cite{vanroos2d,vanroos3d}.
  \item It is well known that the solutions of \emph{nonlinear}
    parabolic equations possibly cease to exist after finite
    time. This is even the case if the nonlinearity only depends on
    the unknown itself instead of its gradient, see e.g.~the classical
    paper~\cite{cohen}. Of course, this is even more so the case if
    the nonlinearity contains gradient dependent terms; we refer
    to~\cite[Ch.~IV]{quittner} and references therein.  Therefore the
    question of \emph{global} existence for the solution in the
    general context of \cref{t-formulat} seems out of
    reach. For related arguments from physics, see~\cite[p.~55]{Marko}.
  \item It is possible to relax the requirements on the initial data
    when working in function spaces with temporal weights,
    see~\cite{PruWi}. Since our impetus was to demonstrate the power
    of the extrapolated regularity result for elliptic operators in a
    real-world problem, this is out of scope here. See
    however~\cite[Thm.~7.1.6]{luna}.
  \end{enumerate}
\end{rem}


%% file: ExtrapolEllipt.bbl
\begin{thebibliography}{10}

\bibitem{Adams_Hedberg} D.~R.~Adams and L.~I.~Hedberg: \emph{Function
  spaces and potential theory}. Grundlehren der mathematischen 
  Wissenschaften Vol.~314, Springer, Berlin (1996).

\bibitem{auscher} P.~Auscher, N.~Badr, R.~Haller-Dintelmann,
  J.~Rehberg: The square root problem for second-order, divergence
  form operators with mixed boundary conditions on
  $L^p$. J.~Evol.~Equ.~15 No.~1 (2015), 165--208.

\bibitem{AuschEg} P.~Auscher, S.~Bortz, M.~Egert, O.~Saari:
   Nonlocal self-improving properties: a functional analytic
  approach.  Tunis.~J.~Math.~1 No.~2 (2019), 151--183.
  
\bibitem{cohen} P.~Baras; L.~Cohen: Complete blow-up after
  $T_{max}$ for the solution of a semilinear heat equation.
  J.~Funct.~Anal.~71 (1987), 142--174.
  
\bibitem{Bec20} S.~Bechtel: Intrinsic characterization of Sobolev
  spaces with boundary conditions. arXiv: 2002.08656 (2020).
  
\bibitem{BM19} S.~Bechtel, M.~Egert: Interpolation theory for Sobolev
  functions with partiallly vanishing trace on irregular open
  sets. J.~Fourier~Anal.~Appl.~25 (2019), 2733-–2781.

\bibitem{Bie09} M.~Biegert: On traces of Sobolev functions on
  the boundary of extension domains. Proc.~Amer.~Math.~Soc.~137
  No.~12 (2009), 4169--4176.

 \bibitem{BGK04} P.~Biler, M.~Guedda,, G.~Karch:
   Asymptotic properties of solutions of the viscous Hamilton-Jacobi
   equation. J.~Evol.~Equ.~4 No.~1 (2004), 75--97.

%

\bibitem{BMMM} K.~Brewster, D.~Mitrea, I.~Mitrea, M.~Mitrea:
  Extending Sobolev Functions with Partially Vanishing Traces
  from Locally ($\epsilon$,$\delta$)-Domains and Applications to
  Mixed Boundary Problems. J.~Funct.~Anal.~266 No.~7 (2014),
  4314--4421.

\bibitem{Cia} P.~G.~Ciarlet: \emph{The finite element method for
    elliptic problems}.  Studies in Mathematics and its
  Applications 4, North-Holland, Amsterdam (1978).

  
  
\bibitem{DHP03} R.~Denk, M.~Hieber, J.~Pr\"uss:
  \emph{$\mathcal{R}$-boundedness, Fourier multipliers and problems of
  elliptic and parabolic type}. Mem.~Am.~Math.~Soc.~166 no.~788 (2003).

\bibitem{disser} K.~Disser, H.-C.~Kaiser, J.~Rehberg: Optimal
  Sobolev regularity for linear second-order divergence elliptic
  operators occurring in real-world problems, SIAM
  J.~Math.~Anal.~47 No.~3 (2015), 1719--1746.

\bibitem{vanroos3d} K.~Disser and J.~Rehberg: The 3D transient
  semiconductor equations with gradient-dependent and interfacial
  recombination. Math.~Models Methods~Appl.~Sci.~29 (2019),
  1819--1851. 

\bibitem{drumm} D.~W.~Drumm, L.~C.~L.~Hollenberg, M.~Y.~Simmons and
  M.~Friesen. Effective mass theory of monolayer $\delta$ doping in
  the high density limit. Phys.~Rev.~B~85 155419 (2012).


\bibitem{E19sys} M.~Egert: $L^p$-estimates for the square root
  of elliptic systems with mixed boundary
  conditions. J.~Differential Equations~265 (2018), 1279--1323. 
  
\bibitem{EHDT} M.~Egert, R.~Haller-Dintelmann, P.~Tolksdorf: The
  Kato square root problem for mixed boundary conditions. 
  J.~Funct.~Anal.~267 No.~5 (2014), 1419--1461.


\bibitem{elschner} J.~Elschner, J.~Rehberg, G.~Schmidt: Optimal
  regularity for elliptic transmission problems including {$C^1$}
  interfaces. Interfaces Free Bound.~9 No.~2 (2007), 233--252.


 \bibitem{tERL19} A.F.M.~ter Elst, J.~Rehberg, A.~Linke: On the
   numerical range of sectorial forms, arXiv:1912.09169 (2019).


 \bibitem{gajewski93} H.~Gajewski: {A}nalysis und {N}umerik von
   {L}adungstransport in {H}alb\-lei\-tern ({A}nalysis and numerics
   of carrier transport in semiconductors),
   Mitt.~Ges.~Angew.~Math.~Mech.~16 No.~1 (1993), 35--57
   (German).

 \bibitem{GGZ} H.~Gajewski, K.~Gr\"oger, K.~Zacharias:
   \emph{Nicht\-lin\-eare Op\-e\-ra\-tor\-gleichungen und
     Op\-e\-ra\-tor\-diff\-erentialgleichungen}. Mathematische
   Lehrb\"ucher und Monographien, II.~Abteilung Mathematische
   Monographien 38. Akademie-Verlag, Berlin (1974). (German)

\bibitem{GGK03}  B.H.~Gilding, M.~Guedda, R.~Kersner:
The Cauchy problem for $u_t = \Delta u + |\nabla
u|^q$. J.~Math.~Anal.~Appl.~284 No.~2 (2003), 733--755.


\bibitem{groeger} K.~Gr\"oger: A $W^{1,p}$-estimate for solutions to
  mixed boundary value problems for second order elliptic differential
  equations. Math.~Ann.~283  No.~4 (1989), 679-–687.

\bibitem{haase} M.~Haase: \emph{The functional calculus for sectorial
  operators}. Operator Theory: Advances and Applications 169.
  Birkh\"auser, Basel (2006).

\bibitem{RoChristJo} R.~Haller-Dintelmann, H.-C.~Kaiser,
  J.~Rehberg. Elliptic model problems including mixed boundary
  conditions and material heterogeneities. J.~Math.~Pures Appl.~89
   No.~1 (2008), 25--48.

\bibitem{RoJo} R.~Haller-Dintelmann, J.~Rehberg: Maximal
  parabolic regularity for divergence operators including mixed
  boundary conditions. J.~Differ.~Equations 247 No.~5 (2009),
  1354--1396.



\bibitem{jonsson} R.~Haller-Dintelmann, A.~Jonsson, D.~Knees,
  J.~Rehberg: Elliptic and parabolic regularity for second order
  divergence operators with mixed boundary conditions. Math.~Methods
  Appl.~Sci.~39 No.~17 (2016), 5007--5026.

\bibitem{haller} R.~Haller-Dintelmann, H.~Meinlschmidt, W.~Wollner:
  Higher regularity for solutions to elliptic systems in divergence
  form subject to mixed boundary conditions. Ann.~Mat.~Pura
  Appl.~198 No.~4 (1923-) (2019), 1227–-1241.

\bibitem{henry} D.~Henry: \emph{Geometric theory of semilinear
    parabolic equations}. Lecture Notes in Mathematics
  840. Springer, Berlin-Heidelberg-New York (1981).


\bibitem{horst} D.~Horstmann, H.~Meinlschmidt, J.~Rehberg: The full
  Keller-Segel model is well-posed on non-smooth domains. Nonlinearity
  Vol.~31 No.~4 (2018), 1560--1592.


\bibitem{joch} F.~Jochmann: A $H^s$ regularity result for the
  gradient of solutions to elliptic equations with mixed boundary
  conditions. J.~Math.~Anal.~Appl.~238 (1999), 429--450.
  

\bibitem{JW84} A.~Jonsson, H.~Wallin: \emph{Function spaces on
    subsets of $\R^n$}.  Harwood Academic Publishers,
  Chur-London-Paris-Utrecht-New York (1984).

\bibitem{vanroos2d} H-C.~Kaiser, H.~Neidhardt, J.~Rehberg:
  Classical solutions of drift-diffusion equations for semiconductor
  devices: The two-dimensional case. Nonlinear Anal., Theory Methods
  Appl., Ser.~A, Theory Methods~71 No.~5-6 (2009), 1584--1605.

\bibitem{kato} T.~Kato: \emph{Perturbation theory for linear
  operator}s. Springer (1976).

\bibitem{landsberg91} P.T.~Landsberg: \emph{{R}ecombination in
    {S}emiconductors}. Cambridge University Press, Cambridge (1991).

\bibitem{luna} A.~Lunardi: \emph{Analytic Semigroups and Optimal
    Regularity in Parabolic Problems}, Birkh\"auser, Basel (1995).

\bibitem{Marko} P.A.~Markowich: \emph{The stationary Semiconductor
    Device Equations}, Springer-Verlag,  Wien (1986).

\bibitem{Marschall} J.~Marschall: Nonregular Pseudo-Differential
  Operators. Z.~Anal.~Anwendungen~15 (1996), 109--148.


\bibitem{MazyaMult} V.~Maz'ya, T.~O.~Shaposhnikova, T.~O.:
  \emph{Theory of Sobolev Multipliers}.  Grundlehren der
  mathematischen Wissenschaften, Springer, Berlin Heidelberg
  (2009).

\bibitem{MMR17a} H.~Meinlschmidt, C.~Meyer, J.~Rehberg: Optimal
  Control of the Thermistor Problem in Three Spatial Dimensions,
  Part 1: Existence of Optimal Controls. SIAM J.~Control
  Optim.~55 No.~5 (2017), 2876–-2904.

\bibitem{MMR17b} H.~Meinlschmidt, C.~Meyer, J.~Rehberg: Optimal
  Control of the Thermistor Problem in Three Spatial Dimensions,
  Part 2: Optimality Conditions. SIAM J.~Control
  Optim.~55 No.~4 (2017), 2368-–2392.

\bibitem{MR16} H.~Meinlschmidt, J.~Rehberg: H\"older-estimates for
  non-autonomous parabolic problems with rough
  data. Evol.~Equ.~Control Theory 5 (2016), 147--184.
  
\bibitem{meyers} N.~G.~Meyers: An $L^p$-estimate for the
  gradient of solutions of second order elliptic divergence
  equations. Ann.~Scuola Norm-Sci.~S\'{e}rie~3 17 No.~3 (1963),
  189--206.

\bibitem{naz} A.~M.~Nazmul, T.~Amemiya, Y.~Shuto, S.~Sugahara, and
  M.~Tanaka: High Temperature Ferromagnetism in GaAs-Based
  Heterostructures with Mn-$\delta$ Doping, Phys. Rev.~Lett.~95 017201
  (2005). Erratum Phys.~Rev.~Lett.~96 149901 (2006).

\bibitem{Ouh05} E.~Ouhabaz: \emph{Analysis of Heat Equations on
    Domains}. 
  London Mathematical Society Monographs Series Vol.~31, Princeton
  University Press, Princeton (2005).
 
\bibitem{pazy} A.~Pazy: \emph{Semigroups of linear operators and
  applications to partial differential equations}. Springer (1983).


\bibitem{Porretta15}  A.~Porretta: Weak solutions to
  Fokker-Planck equations and mean field
  games. Arch.~Ration.~Mech.~Anal.~216 No.~1 (2015), 1--62.

\bibitem{pruess} J.~Pr\"uss, Maximal regularity for evolution
  equations in $L^p$-spaces, Conf.~Semin.~Mat.~Univ.~Bari 285 (2002),
  1--39.

\bibitem{PruessSchnaubelt} J.~Pr\"uss, R.~Schnaubelt:
 Solvability and Maximal Regularity of Parabolic Evolution Equations
 with Coefficients Continuous in Time. J.~Math.~Anal.~Appl.~256
 (2001), 405--430.

\bibitem{PruWi} J.~Pr\"uss, G.~Simonett: Maximal regularity for
  evolution equations in weighted $L^p$-spaces. Arch.~Math.~82,
  No.~5 (2004), 415--431.

\bibitem{quittner} P.~Quittner, P.~Souplet: \emph{Superlinear
  Parabolic Problems}. Birkh\"auser, Basel (2019).
  
\bibitem{RunstSickel} T.~Runst, W.~Sickel: \emph{Sobolev Spaces
    of Fractional Order, Nemytskij Operators, and Nonlinear
    Partial Differential Equations}. Walter de Gruyter, Berlin,
  New York (1996).
  
\bibitem{savare} G.~Savar\'{e}: Regularity Results for Elliptic
  Equations in Lipschitz Domains. J.~Funct.~ Anal.~152 (1998),
  178--201.
  
\bibitem{selberherr84} S.~Selberherr: \emph{Analysis and simulation of
    semiconductor devices}. Springer, Wien (1984).
  
\bibitem{sickel} W.~Sickel: Pointwise multipliers of Lizorkin-Triebel
  spaces. In: The Maz'ya anniversary collection. Op.~Theory,
  Adv.~Appl.~Vol.~110, Birkh\"auser, Basel (1999).

\bibitem{shamir} E.~Shamir: Regularization of mixed
  second-order elliptic problems. Israel~J.~Math.~6 (1968), 150--168.
  
\bibitem{sneib} I.~J.~Sneiberg: Spectral properties of linear
  operators in interpolation families of Banach spaces.
  Mat.~Issled.~9 No.~2 (1974), 214--229.


\bibitem{triebel} H.~Triebel: A note on function spaces in rough
  domains. (English.~Russian original) Proc.~Steklov
  Inst.~Math.~293 (2016),
  338--342. Translation from Tr.~Mat.~Inst.~Steklova 293 (2016),
  346-351.

\bibitem{Triebel} H.~Triebel: \emph{Interpolation theory,
    function spaces, differential operators}. North Holland
  Publishing Company, Amsterdam-New York-Oxford (1978).
  

  
\bibitem{Visvana} K.~Viswanath: \emph{Handbook of Surfaces and
    Interfaces of Materials}. Edited by H.S.~Nalwa. Vol.~1,
  Ch.~3: \emph{Surface and Interfacial Recombination in
  Semiconductors}, Academic Press (2001).

\bibitem{Xiang} J.~Xiang, Y.~Li, F.~Huang, D.~Zhong: Effect
  of interfacial recombination, bulk recombination and carrier
  mobility on the J–V hysteresis behaviors of perovskite solar
  cells: a drift-diffusion simulation study.
  Phys.~Chem.~Chem.~Phys.~21 (2019), 17836--17845.



\end{thebibliography}
